\documentclass[11pt,a4paper]{article}
\usepackage{mathrsfs}
\usepackage{amsfonts}
\usepackage[active]{srcltx}
\usepackage{indentfirst,latexsym,bm,amsmath,pstricks,amssymb,amsthm,amscd}
\usepackage[all]{xy}
\usepackage{tikz,diagbox}
\usepackage{makecell}

\begin{document}
\theoremstyle{plain}
\newtheorem{thm}{Theorem}[section]
\newtheorem{theorem}[thm]{Theorem}
\newtheorem{addendum}[thm]{Addendum}
\newtheorem{lemma}[thm]{Lemma}
\newtheorem{corollary}[thm]{Corollary}
\newtheorem{proposition}[thm]{Proposition}

\newcommand{\mR}{\mathbb{R}}
\newcommand{\mZ}{\mathbb{Z}}
\newcommand{\mN}{\mathbb{N}}
\newcommand{\mC}{\mathbb{C}}
\newcommand{\mP}{\mathbb{P}}
\newcommand{\CO}{\mathcal{O}}
\newcommand{\CE}{\mathcal{E}}
\newcommand{\CF}{\mathcal{F}}
\newcommand{\CG}{\mathcal{G}}
\newcommand{\CL}{\mathcal{L}}
\newcommand{\CM}{\mathcal{M}}
\newcommand{\CP}{\mathcal{P}}
\newcommand{\CS}{\mathcal{S}}
\newcommand{\CA}{\mathcal{A}}
\newcommand{\CB}{\mathcal{B}}
\newcommand{\CC}{\mathcal{C}}
\newcommand{\CH}{\mathcal{H}}
\newcommand{\CI}{\mathcal{I}}
\newcommand{\CJ}{\mathcal{J}}
\newcommand{\CZ}{\mathcal{Z}}
\newcommand{\CN}{\mathcal{N}}
\newcommand{\CR}{\mathcal{R}}

\newcommand{\FA}{\mathfrak{A}}
\newcommand{\FB}{\mathfrak{B}}
\newcommand{\FC}{\mathfrak{C}}
\newcommand{\FD}{\mathfrak{D}}

\newcommand{\SA}{\mathscr{A}}
\newcommand{\SB}{\mathscr{B}}
\newcommand{\SC}{\mathscr{C}}
\newcommand{\SD}{\mathscr{D}}

\newcommand{\mr}{\mbox}
\newcommand{\I}{{\bf I}}
\newcommand{\J}{{\bf J}}
\newcommand{\bs}{{\bf s}}
\newcommand{\ts}{{\tilde s}}
\newcommand{\e}{{{\bf e}}}
\newcommand{\m}{{\bf m}}
\newcommand{\ord}{\mathrm{ord}}
\newcommand{\fb}{\mathfrak{b}}

\theoremstyle{definition}
\newtheorem{remark}[thm]{Remark}
\newtheorem{remarks}[thm]{Remarks}
\newtheorem{notations}[thm]{Notations}
\newtheorem{definition}[thm]{Definition}
\newtheorem{claim}[thm]{Claim}
\newtheorem{assumption}[thm]{Assumption}
\newtheorem{assumptions}[thm]{Assumptions}
\newtheorem{property}[thm]{Property}
\newtheorem{properties}[thm]{Properties}
\newtheorem{example}[thm]{Example}
\newtheorem{examples}{Examples}
\newtheorem{conjecture}[thm]{Conjecture}
\newtheorem{questions}[thm]{Questions}
\newtheorem{question}[thm]{Question}
\newtheorem{problem}[thm]{Problem}
\numberwithin{equation}{section}
 \newcommand{\Rnm}[1]{\uppercase\expandafter{\romannumeral #1}}

 \title{On the sharp lower bounds of  modular invariants and fractional Dehn twist coefficients}
\author{Xiao-Lei Liu~~~~Sheng-Li Tan}

\date{}

 \maketitle

{\bf Abstract} { Modular invariants of families of curves are Arakelov invariants in arithmetic algebraic geometry. All the known uniform lower bounds of these invariants are not sharp. In this paper, we aim to give explicit lower bounds of modular invariants of families of curves, which is sharp for genus 2. According to the relation between fractional Dehn twists and modular invariants, we give the sharp lower bounds of fractional Dehn twist coefficients and classify  pseudo-periodic maps with minimal  coefficients for genus 2 and 3 firstly.  We also obtain a rigidity property for families with minimal modular invariants, and other applications.}\\

{\bf{Keywords}} {lower bounds, fractional Dehn twists, modular invariants}\\

{\bf{MSC(2010)}} {14D06, 14H10, 57M50, 57M99}

\section{Introduction}

\subsection{Modular invariants}\label{sectBasicResult}

Without mention we always work on complex number field $\mathbb{C}$.  A {\it family} of projective curves of genus $g$ is a surjective holomorphic morphism $f:S\to C$ whose general fiber is a smooth curve of genus $g$, where $S$ is a smooth projective surface, and $C$ is a smooth projective curve of genus $b$. Let ${\CM_g}$ be the moduli space of smooth curves of genus $g$, and $\overline{\CM}_g$ be the Deligne-Mumford compactification of $\CM_g$.

 The intersection theory of divisors of $\overline{\CM}_g$ is very beautiful.  The intersection of rational divisor class $\gamma$ of $\overline{\CM}_g$ with curves $D\subset\overline{\CM}_g$ is also interested in the theory of birational geometry of $\overline{\CM}_g$.  Numerically, the intersection of $\gamma$ with $D$  can be regarded as the degree of $\gamma$ on $D$.

 The {\it modular invariant} of $f$ corresponding to $\gamma$ is defined as the degree $\gamma(f)=\mathrm{deg}J_f^*(\gamma)$,  where $J_f:C\to  \overline{\CM}_g$ is the induced moduli map (see \cite{Ta10}). The modular invariant $\gamma(f)$ satisfies the {\it base change property}, i.e., if
$\tilde f:\tilde S\to \tilde C$ is the pullback fibration of $f$ under a base change $\pi:\tilde C \to C$ of degree $d$, then
$\gamma(\tilde f)=d\cdot \gamma(f)$. Modular invariants are also important in many other mathematical branches, such as arithmetic geometry (\cite{Ja14}), low-dimensional topology (\cite{Liu}), and ordinary differential equations (\cite{Ta}). In particular, the modular invariants are generalized to be new  invariants, Chern numbers, of ordinary differential equations (\cite{Ta}) by the second author.

Let $\lambda$ be the Hodge divisor class of $\overline{\CM}_g$, $\delta$ be the boundary divisor class, and their corresponding modular invariants be $\lambda(f)$ and $\delta(f)$. We also denote by $\kappa(f)$ the modular invariant corresponding to $\kappa=12\lambda-\delta$. By stable reduction theorem (see \cite{Xi90}), we know that all these modular invariants are nonnegative rational numbers, and $\lambda(f)>0$ for non-isotrivial family $f$.
These kinds of modular invariants are called Arakelov invariants (\cite{Ja14}) in number field case, and the modular invariant $\lambda(f)$ is  Faltings height in particular.

The minimal uniform lower bounds for these invariants are interesting. In 1991, Mazur raised a question on the minimized Faltings heights of varieties in $\mathbb{P}^N$, which was studied by Zhang (\cite{Zh96}) partly.  Some uniform lower bounds for Faltings $\delta$-invariant of curves are also obtained by Faltings et al (\cite{Fa84,Wi16}...). But all these known bounds are not sharp.

In this paper, we consider the above uniform lower bounds problem in the case of curves over function fields. Our goal is to get sharp lower bounds depending only on $g$ and characterize families with these minimal lower bounds.

When $g=1$, then the best lower bounds are $\kappa(f)=0$, $\lambda(f)\geq\frac1{12}$ and $\delta(f)\geq1$.  We believe that these bounds of modular invariants are known to experts. Hence we only considered $g\geq 2$.

For $g=2$, we have the following sharp lower bounds of modular invariants.

\begin{theorem}\label{thmlambdakappa}
Let $f:S\to C$ be a non-isotrivial fibration of genus 2, then
\begin{equation}\label{equationlambda}
\lambda(f)\geq\frac{1}{60},~ \kappa(f)\geq\frac{1}{15},~\delta(f)\geq\frac{1}{12},
\end{equation}
and each equality can be reached. Furthermore,

1)  $\lambda(f)=\frac{1}{60}$ if and only if $\delta(f)=\frac{1}{12}$
if and only if all the singular fibers of $f$ have smooth reduction except one whose dual graph is either Figure (2-1a) or Figure (2-1b).

{\upshape
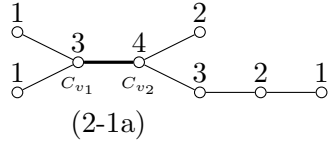
\begin{center}
\begin{tikzpicture}
  [inner sep=0.5mm,
  place/.style={circle,draw}]
  \node[place] (v1) at (0.8,0.8) [label=above:3]  {};
    \node[place] (v2) at (1.6,0.8) [label=above:4]  {};
    \node[place] (v3) at (0,1.2) [label=above:1]  {};
  \node[place] (v4) at (0,0.4) [label=above:1]  {};
    \node[place] (v5) at (2.4,1.2) [label=above:2]  {};
  \node[place] (v6) at (2.4,0.4) [label=above:3]  {};
   \node[place] (v7) at (3.2,0.4) [label=above:2]  {};
    \node[place] (v8) at (4.0,0.4) [label=above:1]  {};
  \draw (v3)--(v1)--(v2)--(v5);\draw (v4)--(v1)-- (v2)--(v6)--(v7)--(v8);
  \draw[very thick] (v1)--(v2);
    \node [below=4pt] at (v2)  {\tiny $C_{v_2}$};
          \node [below=4pt] at (v1)  {\tiny $C_{v_1}$};
   \node at (1.2,0.5) [below=8pt] {(2-1a)};
\end{tikzpicture}
~~~~~~~~
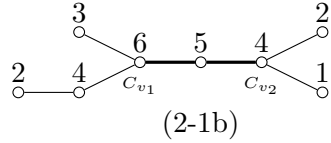
\begin{tikzpicture}
  [inner sep=0.5mm,
  place/.style={circle,draw}]
  \node[place] (v1) at (0.8,0.8) [label=above:6]  {};
  \node[place] (v3) at (1.6,0.8) [label=above:5]  {};
    \node[place] (v2) at (2.4,0.8) [label=above:4]  {};
  \node[place] (v5) at (-0.8,0.4) [label=above:2]  {};
    \node[place] (v4) at (0,1.2) [label=above:3]  {};
  \node[place] (v6) at (0,0.4) [label=above:4]  {};
    \node[place] (v7) at (3.2,1.2) [label=above:2]  {};
  \node[place] (v8) at (3.2,0.4) [label=above:1]  {};
  \draw (v4)-- (v1)--(v3)--(v2)--(v7);\draw (v5)--(v6)--(v1)--(v3)-- (v2)--(v8);
  \draw[very thick] (v1)--(v3)--(v2);
    \node [below=4pt] at (v2)  {\tiny $C_{v_2}$};
          \node [below=4pt] at (v1)  {\tiny $C_{v_1}$};
   \node at (1.6,0.5) [below=8pt] {(2-1b)};
\end{tikzpicture}
\end{center}}

2) $\kappa(f)=\frac1{15}$ if and only if all the singular fibers of $f$ have smooth reduction except one whose dual graph is either Figure (2-0a) or Figure (2-0b).

\vspace{1mm}
{\upshape
\begin{center}
  \begin{tikzpicture}
  [  place/.style={circle,draw,inner sep=0.5mm},
  place2/.style={circle,draw,fill,inner sep=0.2mm},]
  \node[place] (v1) at (0.8,0.8) [label=above:2]  {};
    \node[place] (v2) at (1.6,0.8) [label=above:6]  {};
    \node[place] (v3) at (0,1.2) [label=above:1]  {};
  \node[place] (v4) at (0,0.4) [label=above:1]  {};
  \node[place] (v5) at (2.4,1.2) [label=above:1]  {};
   \node[place] (v6) at (2.4,0.4) [label=above:3]  {};
  \draw (v3)-- (v1)--(v2)--(v5);\draw (v4)--(v1)--(v2)--(v6);
   \node at (1.2,0.5) [below=8pt] {(2-0a)};
    \draw[very thick]  (v1)  to (v2); \node [below=4pt] at (v2)  {\tiny $C_{v}$};
\end{tikzpicture}
~~~~~~~~~~
\begin{tikzpicture}
  [inner sep=0.5mm,
  place/.style={circle,draw}]
  \node[place] (v1) at (0.8,0.8) [label=above:3]  {};
    \node[place] (v2) at (1.6,0.8) [label=above:2]  {};
  \node[place] (v3) at (0,0.8) [label=above:2]  {};
   \node[place] (v4) at (-0.8,0.8) [label=above:1]  {};
  \draw[very thick]  (v1) to [bend left=15]   (v2);
  \draw[very thick] (v1) to [bend right=15]  (v2);
  \draw (v4) to (v3) to (v1); \node [below=4pt] at (v1)  {\tiny $C_{v}$};
  \node at (0.4,0.5) [below=8pt] {(2-0b)};
\end{tikzpicture}
\end{center}
}

\end{theorem}

We prove that the lower bounds in Theorem \ref{thmlambdakappa} are optimum by giving examples in Section \ref{sectionexistence}.

\begin{theorem}\label{thmexistence}
There exists a family of fibrations $(f_{\lambda,n}:S_{\lambda,n}\to \mP^1)_{n\in\mN}$ (resp. $(f_{\kappa,n}:S_{\kappa,n}\to \mP^1)_{n\in\mN}$) of genus 2 with $\lambda(f_{\lambda,n})=\frac1{60},\delta(f_{\lambda,n})=\frac1{12}$ (resp. $\kappa(f_{\kappa,n})=\frac1{15},~\lambda(f_{\kappa,n})=\frac1{30},~\delta(f_{\kappa,n})=\frac1{3}$),  satisfying that

1) $f_{\lambda,n}$ (resp. $f_{\kappa,n}$) has $2n+3$ singular fibers;

2) the  image of $f_{\lambda,n}$ (resp. $f_{\kappa,n}$) in $\overline{\CM}_g$ by the moduli  map $J:\mP^1\to \overline{\CM}_g$ is the same as that of $f_{\lambda,0}$ (resp. $f_{\kappa,0}$), for each $n\in\mN$.

Moreover,  if $n=0$, then $S_{\lambda,0}$ and $S_{\kappa,0}$ are both rational surfaces.

\end{theorem}

Furthermore, if the fibred surface is rational, then we have the following rigidity property.
\begin{theorem}\label{thmrigidity}
 There are only finitely many  fibrations $f:S\to C$ of genus 2 such that $S$ is a rational surface $S$, and $\lambda(f)=\frac1{60}$.
\end{theorem}

For  $g=3$, we have the following results.

\begin{theorem}\label{thmg3mod}
  Let $f:S\to C$ be a fibration of genus 3 with $\delta(f)\neq 0$, then
$$\lambda(f)\geq\frac1{105},~~~\delta(f)\geq \frac1{30},~~~\kappa (f)\geq\frac8{315}.$$
Moreover,  $\delta(f)=\frac1{30}$ if and only if all singular fibers of $f$ have periodic monodromy except one whose dual graph is one of figures (3-1a), (3-1b) and (3-1c) as follows.

\vspace{2mm}
{\upshape
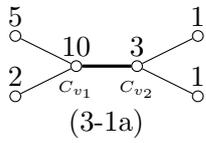
\begin{center}
\begin{tikzpicture}
  [inner sep=0.5mm,
  place/.style={circle,draw}]
  \node[place] (v1) at (0.8,0.8) [label=above:10]  {};
    \node[place] (v2) at (1.6,0.8) [label=above:3]  {};
  \node[place] (v3) at (0,1.2) [label=above:5]  {};
  \node[place] (v4) at (0,0.4) [label=above:2]  {};
    \node[place] (v5) at (2.4,1.2) [label=above:1]  {};
  \node[place] (v6) at (2.4,0.4) [label=above:1]  {};
  \draw (v3)-- (v1)--(v2)--(v5);\draw (v4)--(v1);\draw (v2)--(v6);\draw[very thick] (v1)--(v2);
  \node [below=4pt] at (v2)  {\tiny $C_{v_2}$};
          \node [below=4pt] at (v1)  {\tiny $C_{v_1}$};
  \node at (1.2,0.5) [below=6pt] {(3-1a)};
\end{tikzpicture}
~~~~
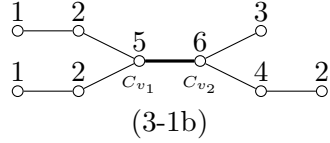
\begin{tikzpicture}
  [inner sep=0.5mm,
  place/.style={circle,draw}]
  \node[place] (v1) at (0.8,0.8) [label=above:5]  {};
    \node[place] (v2) at (1.6,0.8) [label=above:6]  {};
  \node[place] (v3) at (-0.8,1.2) [label=above:1]  {};
  \node[place] (v5) at (-0.8,0.4) [label=above:1]  {};
    \node[place] (v4) at (0,1.2) [label=above:2]  {};
  \node[place] (v6) at (0,0.4) [label=above:2]  {};
    \node[place] (v7) at (2.4,1.2) [label=above:3]  {};
  \node[place] (v8) at (2.4,0.4) [label=above:4]  {};
   \node[place] (v9) at (3.2,0.4) [label=above:2]  {};
  \draw (v3)--(v4)-- (v1)--(v2)--(v7);\draw (v5)--(v6)--(v1)-- (v2)--(v8)--(v9);
  \draw[very thick] (v1)--(v2);
    \node [below=4pt] at (v2)  {\tiny $C_{v_2}$};
          \node [below=4pt] at (v1)  {\tiny $C_{v_1}$};
   \node at (1.2,0.5) [below=8pt] {(3-1b)};
\end{tikzpicture}
~~~~
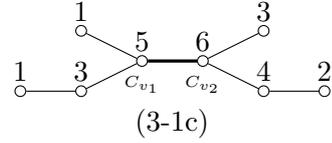
\begin{tikzpicture}
  [inner sep=0.5mm,
  place/.style={circle,draw}]
  \node[place] (v1) at (0.8,0.8) [label=above:5]  {};
    \node[place] (v2) at (1.6,0.8) [label=above:6]  {};
  \node[place] (v5) at (-0.8,0.4) [label=above:1]  {};
    \node[place] (v4) at (0,1.2) [label=above:1]  {};
  \node[place] (v6) at (0,0.4) [label=above:3]  {};
    \node[place] (v7) at (2.4,1.2) [label=above:3]  {};
  \node[place] (v8) at (2.4,0.4) [label=above:4]  {};
   \node[place] (v9) at (3.2,0.4) [label=above:2]  {};
  \draw (v4)-- (v1)--(v2)--(v7);\draw (v5)--(v6)--(v1)-- (v2)--(v8)--(v9);
  \draw[very thick] (v1)--(v2);
    \node [below=4pt] at (v2)  {\tiny $C_{v_2}$};
          \node [below=4pt] at (v1)  {\tiny $C_{v_1}$};
   \node at (1.2,0.5) [below=8pt] {(3-1c)};
\end{tikzpicture}
\end{center}
}
\end{theorem}

  Now, we will show that when $g=3$, $\lambda(f)=\frac1{105}$ can be ``combinatorially reached".

  If all singular fibers of $f$ are $F_1,\cdots,F_s$, then we call ($F_1$,$\cdots$,$F_s$) the {\it configuration} of singular fibers of $f$. Denote by $F_a$ (resp. $F_b$) the singular fiber whose dual graph is Figure (3-1a) (resp. (3-1b)), and by (i2) the corresponding singular fiber in \cite[p.202]{AI02}. Then $F_a$ (resp. $F_b$) is {\it hyperelliptic},  that is, $F_a$ (resp. $F_b$) can be realized as a singular fiber of some hyperelliptic fibration (\cite[pp.19-20]{Is04}).
\begin{theorem}\label{thmg=3hyp}
   Let $S$ be a rational surface, and $f:S\to C$ be a hyperelliptic fibration of genus 3. Assume that either $F_a$ or $F_b$ is a singular fiber of $f$. Then $\lambda(f)=\frac1{105}$ if and only if the configuration of singular fibers of $f$  is one of the following:

{\upshape
   ~($F_{a}$,(i2),(i4)), ~~~~~~($F_{a}$,(i6),(i6)), ~~~~($F_{a}$,(i26),(i44)), ~~

($F_{a}$,(i26),(i45)), ~~($F_{a}$,(i26),(i46)), ~~($F_{a}$,(i26),(i47),(i47)), ~~~ ($F_{b}$,(i26),(i47)).
}

\end{theorem}

It is natural to consider which configuration in Theorem \ref{thmg=3hyp} can be realized. If $S$ is not rational, we will obtain more possible combinatorial configurations (see the proof of Theorem \ref{thmg=3hyp}), so we give the following conjecture.
\begin{conjecture}
  There exists a hyperelliptic fibration $f$ of genus 3 with $\lambda(f)=\frac1{105}$.
\end{conjecture}

For any $g\geq4$, similarly as $g=3$, it is possible to obtain the  sharp lower bound  of $\lambda(f)$, see Remark \ref{remg4}.
For uniform lower bounds for $g\geq4$, we have the following lower bounds depending only on $g$.
\begin{theorem}\label{thmg4mod}
Suppose $f:S\to C$ is a  fibration of genus $g\geq4$, and
$\delta(f)\neq0$, then
 $$\lambda(f)\geq\frac{1}{16g(2g+1)},~~
 \delta(f)\geq\frac1{4(g+1)^2},~~\kappa(f)\geq\frac{g-1}{4g^2(2g+1)}.$$
\end{theorem}

Let $\lambda(g)$ be the sharp lower bound of $\lambda(f)$ for non-isotrivial families of curves $f$ of genus $g$. From the above, we know that $\lambda(1)=\frac1{12}$, $\lambda(2)=\frac1{60}$. Since modular invariants are heights in arithmetic algebraic geometry,  we raise the following effective question which relates to finiteness of points on curves.
\begin{question}
 Is there a positive real number $r_0>0$ with
$$\liminf_{g\to\infty}\lambda(g)\geq r_0?$$
\end{question}

\subsection{Fractional Dehn twist coefficients}\label{subsectInt-FDTC}

It is proved that the modular invariant $\delta(f)$ is a summation of fractional Dehn twist coefficients (\cite{Liu}). So to get results in Section \ref{sectBasicResult}, we need sharp lower bounds of fractional Dehn twist coefficients, which is an interesting problem in low-dimensional topology.

It is known that Dehn twists are the generators of the mapping class group, and fractional Dehn twist coefficients are also important in 3-manifolds. These coefficients were first studied by Gabai and Oertel in \cite{GO89}, and then applied in many aspects  (\cite{HKM07,HM18}...). The bounds of these coefficients are studied in many different contexts, see \cite[Theorem 1]{HM18}, \cite[Section 7]{IK12}, \cite[Theorem 2.16]{KR13},\cite[Theorem 1.5]{Liu}.

For our purpose, we consider fractional Dehn twists coefficents in pseudo-periodic maps, and  try to give their sharp uniform lower bounds which depends only on $g$. Before we state our results, we will introduce some notations first.

Let $\Sigma_g$ be a closed connected Riemann surface of genus $g\geq2$. The mapping class group $\mbox{Mod}(\Sigma_g)$ of $\Sigma_g$ is the group of isotopy classes of orientation preserving homeomorphism of $\Sigma_g$.  The Nielsen-Thurston classification theorem says that any mapping class $\phi\in \mbox{Mod}(\Sigma_g)$ is either periodic, pseudo-Anosov, or reducible. The homeomorphism $\phi$ is reducible if there exist  finite simple closed curves $\SC=\{\gamma_1,\ldots,\gamma_r\}$ on $\Sigma_g$ such that the restriction of $\phi$ on  $\Sigma_g-\SC$ is either periodic or pseudo-Anosov. If $\phi\in \mathrm{Mod}(\Sigma_g)$ is periodic, or $\phi$ is reducible and the restriction is periodic, then $\phi$ is said to be {\it pseudo-periodic}. We may assume $\SC$ satisfies the following additional conditions: (i) $\gamma_i$ does not bound a disk on $\Sigma_g$, and (ii) $\gamma_i$ and $\gamma_j$ are disjoint, and $\gamma_i$ is not parallel to $\gamma_j$ if $i\neq j$ (\cite[Lemma 1.1]{MM11}). Such $\SC$ is called an {\it admissible system} of cut curves.

  Given a pseudo-periodic map $\phi$, a sufficiently high power $\phi^m$ preserves each cut curve $\gamma_1,\ldots,\gamma_r$.   Denote by $T_{\gamma_i}$ the (right-hand) Dehn twist of $\Sigma_g$ along $\gamma_i$,  then there is a factorization of $\phi$ into a  commutative product
  $\phi^m=T_{\gamma_1}^{k_1}\cdots T_{\gamma_r}^{k_r}.$
 The {\it fractional Dehn twist coefficient} of $\phi$ along $\gamma_i$ is defined to be
  $c(\phi,\gamma_i)=k_i/m$ (see \cite[Section 2.2.2]{Li17}).

 If $\phi\in {\mathrm{Mod}}(\Sigma_g)$ is a pseudo-periodic map {\it  of negative twist}, that is, $c(\phi,\gamma)<0$ for each $\gamma\in\SC$, then there exists a local family $f_\phi:S\to\Delta$ whose monodromy homeomorphism around its central fiber is equal (up to isotopy and conjugation) to $\phi$ (\cite{MM11,Im09, Ta01}). Here, the local family $f_\phi:S\to\Delta$ means a  proper surjective holomorphic map from a complex surface $S$ to the unit disk of the complex plane $\Delta$, and only the central fiber $F_\phi=f_\phi^{-1}(0)$ over the origin is singular. We also call $F_\phi$ the singular fiber of $\phi$.

It is known that the topological types of local families $f$ of genus $g\geq2$ are 1-1 correspondent to the conjugacy classes of pseudo-periodic maps $\phi$ of negative twist \cite[Theorem 0.2]{MM11}. Almost all the topological types of local families can be determined by dual graphs of their central fibers. (It is easy to check that the dual graphs in this paper determine topological types of the corresponding local families.)  Hence we denote the conjugacy class of $\phi$ by the dual graph $G(F_\phi)$ of $F_\phi$ for simplicity.

Now we give the sharp lower bounds of fractional Dehn twist coefficients ($|c(\phi,\gamma)|$ in fact), and classify the pseudo-periodic maps with these bounds.

\begin{theorem}\label{thmFDCT}
  Let $\phi\in {\mathrm{Mod}}(\Sigma_g)$ be a pseudo-periodic map of negative twist, and $\gamma\in\SC$.

(1) If $g=2$, then
$$|c(\phi,\gamma)|\geq \frac1{12},$$
and the equality holds if and only if $(G(F_\phi),\gamma)$ is either Figure (2-1a) or Figure (2-1b).

In each figure, we use thick edges  to denote $\gamma$ using the correspondence in \cite[Theorem 4.2]{Liu}.

(2) If $g=3$, then
$$|c(\phi,\gamma)|\geq \frac1{30},$$
and the equality holds if and only if $(G(F_\phi),\gamma)$ is one of figures (3-1a), (3-1b) and (3-1c).
\end{theorem}

Moreover, we can obtain more precise results.

If $\gamma$ is a non-separated cut curve, then $\gamma$ is said to be {\it of type 0}. If $\gamma$ is separated, and the least genus  of the two connected components is $i\geq1$, then $\gamma$ is said to be {\it of type $i$}.

\begin{theorem}\label{thmFDCTi}
  Let $\phi\in {\mathrm{Mod}}(\Sigma_g)$ be a pseudo-periodic map of negative twist, and $\gamma\in\SC$. Then the sharp lower bounds $c$ of $|c(\phi,\gamma)|$ are as follows.

\begin{center}
\begin{tabular}{|c|c|c|}

  \hline
\diagbox{$g$}{$c$}{\mbox{type}} & \mbox{type~ 0} & \mbox{type~1}\\
  \hline
  {\footnotesize$g=2$} & {\footnotesize$\frac13$} & {\footnotesize$\frac1{12}$}\\
  \hline
  {\footnotesize$g=3$} & {\footnotesize$\frac1{12}$} & {\footnotesize$\frac1{30}$}\\
  \hline
\end{tabular}
\end{center}

Furthermore,

 {\upshape(2-0)} if $g=2$, $\gamma$ is of type 0, then $|c(\phi,\gamma)|= \frac1{3}$ if and only if $(G(F_\phi),\gamma)$ is one of the four figures: Figure (2-0a) - Figure (2-0d).

 {\upshape
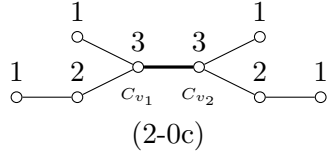
\begin{center}
\begin{tikzpicture}
  [  place/.style={circle,draw,inner sep=0.5mm},
  place2/.style={circle,draw,fill,inner sep=0.2mm},]

  \node[place] (v1) at (0.8,0.8) [label=above:3]  {};
    \node[place] (v2) at (1.6,0.8) [label=above:3]  {};
  \node[place] (v5) at (-0.8,0.4) [label=above:1]  {};
    \node[place] (v4) at (0,1.2) [label=above:1]  {};
  \node[place] (v6) at (0,0.4) [label=above:2]  {};
  \node[place] (v8) at (2.4,1.2) [label=above:1]  {};
    \node[place] (v10) at (3.2,0.4) [label=above:1]  {};
   \node[place] (v9) at (2.4,0.4) [label=above:2]  {};
  \draw (v4)-- (v1)--(v2)--(v8);\draw (v5)--(v6)--(v1); \draw  (v2)--(v9)--(v10);
   \node at (1.2,0.5) [below=8pt] {(2-0c)};
   \draw[very thick]  (v1)  to (v2); \node [below=4pt] at (v2)  {\tiny $C_{v_2}$}; \node [below=4pt] at (v1)  {\tiny $C_{v_1}$};
\end{tikzpicture}
~~~~~~
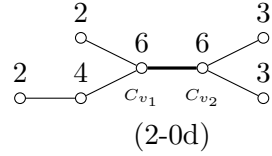
\begin{tikzpicture}
  [  place/.style={circle,draw,inner sep=0.5mm},
  place2/.style={circle,draw,fill,inner sep=0.2mm},]

  \node[place] (v1) at (0.8,0.8) [label=above:6]  {};
    \node[place] (v2) at (1.6,0.8) [label=above:6]  {};
  \node[place] (v5) at (-0.8,0.4) [label=above:2]  {};
    \node[place] (v4) at (0,1.2) [label=above:2]  {};
  \node[place] (v6) at (0,0.4) [label=above:4]  {};
  \node[place] (v8) at (2.4,1.2) [label=above:3]  {};
   \node[place] (v9) at (2.4,0.4) [label=above:3]  {};
  \draw (v4)-- (v1)--(v2)--(v8);\draw (v5)--(v6)--(v1); \draw  (v2)--(v9);
   \node at (1.2,0.5) [below=8pt] {(2-0d)};
   \draw[very thick]  (v1)  to (v2);  \node [below=4pt] at (v1)  {\tiny $C_{v_1}$};  \node [below=4pt] at (v2)  {\tiny $C_{v_2}$};
\end{tikzpicture}
\end{center}}

{\upshape(2-1)} if $g=2$, $\gamma$ is of type 1, then $|c(\phi,\gamma)|= \frac1{12}$ if and only if $(G(F_\phi),\gamma)$ is either Figure (2-1a) or Figure (2-1b).

{\upshape(3-0)} if $g=3$, $\gamma$ is of type 0, then $|c(\phi,\gamma)|= \frac1{12}$ if and only if $(G(F_\phi),\gamma)$ is one of the following three figures
\vspace{0mm}

{\upshape
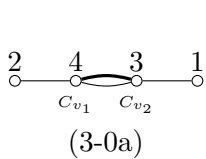
\begin{center}
\begin{tikzpicture}
  [inner sep=0.5mm,
  place/.style={circle,draw}]
  \node[place] (v1) at (0.8,0.8) [label=above:4]  {};
    \node[place] (v2) at (1.6,0.8) [label=above:3]  {};
  \node[place] (v3) at (0,0.8) [label=above:2]  {};
  \node[place] (v4) at (2.4,0.8) [label=above:1]  {};
  \draw[very thick]  (v1) to [bend left=15]   (v2);
  \draw (v1) to [bend right=15]  (v2);
  \draw (v3) to (v1);\draw (v2) to (v4);
\node [below=4pt] at (v2)  {\tiny $C_{v_2}$}; \node [below=4pt] at (v1)  {\tiny $C_{v_1}$};
  \node at (1.2,0.5) [below=8pt] {(3-0a)};
\end{tikzpicture}
~~
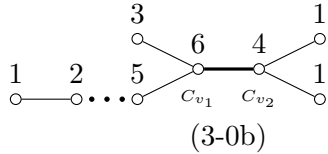
\begin{tikzpicture}
  [  place/.style={circle,draw,inner sep=0.5mm},
  place2/.style={circle,draw,fill,inner sep=0.2mm},]
  \node[place2]  at (-0.6,0.4)  {}; \node[place2]  at (-0.4,0.4)  {};
  \node[place2]  at (-0.2,0.4)  {};
  \node[place] (v1) at (0.8,0.8) [label=above:6]  {};
    \node[place] (v2) at (1.6,0.8) [label=above:4]  {};
  \node[place] (v5) at (-0.8,0.4) [label=above:2]  {};
    \node[place] (v4) at (0,1.2) [label=above:3]  {};
  \node[place] (v6) at (0,0.4) [label=above:5]  {};
    \node[place] (v7) at (-1.6,0.4) [label=above:1]  {};
  \node[place] (v8) at (2.4,1.2) [label=above:1]  {};
   \node[place] (v9) at (2.4,0.4) [label=above:1]  {};
  \draw (v4)-- (v1)--(v2)--(v8);\draw (v7)--(v5);\draw (v6)--(v1); \draw  (v2)--(v9);
   \node at (1.2,0.5) [below=8pt] {(3-0b)};
   \draw[very thick]  (v1)  to (v2); \node [below=4pt] at (v2)  {\tiny $C_{v_2}$}; \node [below=4pt] at (v1)  {\tiny $C_{v_1}$};
\end{tikzpicture}
~~
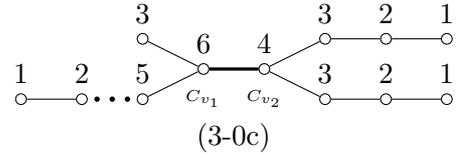
\begin{tikzpicture}
  [  place/.style={circle,draw,inner sep=0.5mm},
  place2/.style={circle,draw,fill,inner sep=0.2mm},]
  \node[place2]  at (-0.6,0.4)  {}; \node[place2]  at (-0.4,0.4)  {};
  \node[place2]  at (-0.2,0.4)  {};
  \node[place] (v1) at (0.8,0.8) [label=above:6]  {};
    \node[place] (v2) at (1.6,0.8) [label=above:4]  {};
  \node[place] (v5) at (-0.8,0.4) [label=above:2]  {};
    \node[place] (v4) at (0,1.2) [label=above:3]  {};
  \node[place] (v6) at (0,0.4) [label=above:5]  {};
    \node[place] (v7) at (-1.6,0.4) [label=above:1]  {};
  \node[place] (v8) at (2.4,1.2) [label=above:3]  {};
   \node[place] (v9) at (2.4,0.4) [label=above:3]  {};
   \node[place] (v10) at (3.2,1.2) [label=above:2]  {};
   \node[place] (v11) at (3.2,0.4) [label=above:2]  {};
   \node[place] (v12) at (4,1.2) [label=above:1]  {};
   \node[place] (v13) at (4,0.4) [label=above:1]  {};
  \draw (v4)-- (v1)--(v2)--(v8);\draw (v7)--(v5);\draw (v6)--(v1); \draw  (v2)--(v9);\draw (v8)--(v10)--(v12); \draw (v9)--(v11)--(v13);
   \node at (1.2,0.5) [below=8pt] {(3-0c)};
    \draw[very thick]  (v1)  to (v2); \node [below=4pt] at (v2)  {\tiny $C_{v_2}$}; \node [below=4pt] at (v1)  {\tiny $C_{v_1}$};
\end{tikzpicture}

\end{center}}

{\upshape(3-1)} if $g=3$, $\gamma$ is of type 1, then $|c(\phi,\gamma)|= \frac1{30}$ if and only if  $(G(F_\phi),\gamma)$ is one of figures { (3-1a)}, { (3-1b)} and { (3-1c)}.

\end{theorem}

Remark that there are two edges between $C_{v_1}$ and $C_{v_2}$ in Figure (3-0a), and each edge corresponds to a  cut curve $\gamma$ satisfying Theorem \ref{thmFDCTi} (3-0). We only label one by a thick line for simplicity. For comparison, there are two edges between $C_v$ and the vertex $C'$ on the right of $C_v$ in Figure (2-0b). It is easy to see that $C'$  does not correspond to a connected component of $\Sigma_g-\SC$, and the two edges together correspond to a cut curve (\cite{Liu}).

For general $g\geq4$, there is a uniform lower bound of fractional Dehn twist coefficients in \cite[Theorem 1.5]{Liu} which is not sharp.

\subsection {Effective Bogomolov conjecture}
Though we have given a uniform lower bound of the effective Bogomolov conjecture for general $g$, see \cite{LT17}. We now give a better bound for $g=2,3$.

Fix an algebraically closed field $k$ of characteristic zero and a
smooth proper connected curve $Y/k$. Define $K$ to be the field of
rational functions on $Y$. Let $C$ be a smooth proper geometrically
connected curve of genus at least 2 over the function field $K$. Denote by $f:X\to Y$ the minimal regular model of the curve $C$ over $Y$, where $X$ is a smooth projective surface over $k$.
Choose a divisor $D$ of degree 1 on $\bar C=C\times_K\bar K$ and
consider the embedding of $C$ into its Jacobian
$\mbox{Jac}(\bar C)=\mbox{Pic}^0(\bar C)$ given on geometric points by
$j_D(x)=[x]-D$. Define
$$a'(D)=\liminf_{x\in C(\bar{K})}\hat{h}(j_D(x)),$$
where $\hat h$ is the canonical N\'eron-Tate height on the Jacobian associated to the symmetric ample divisor $\Theta+[-1]^*\Theta$. As $C(\bar K)$ may not be countable, the liminf is taken to mean the limit over the directed set of all cofinite subsets of $C(\bar K)$ of the infimum of the heights of points in such a subset.

\begin{theorem}\label{thmBogo}
Let $C/K$ be a smooth proper geometrically connected curve of genus $g
\geq2$, if the semistable reduction of $C$ is not smooth,  then
$$\inf_{D\in\mr{Div}^1(\bar{C})}a'(D)\geq
\begin{cases}
\frac{1}{2280}, & g=2;\\
\frac{1}{3276}, & g=3.
\end{cases}$$
\end{theorem}
Remark that the bounds given in \cite{LT17} are $\frac1{12160}$ ~for $g=2$ and $\frac1{19656}$ for $g=3$.

The organization of this paper is as follows.

In \S \ref{sectprel}, we give notations of modular invariants $\delta_i(f)$ and valencies of periodic maps, which will be used in our proofs. In \S\ref{sectFDCTbounds}, we obtain the sharp lower bounds of frational Dehn twist coefficients (Theorem \ref{thmFDCT} and Theorem \ref{thmFDCTi}), using the theory of classification of singular fibers in \cite{AI02}. We divide Theorem \ref{thmFDCTi} into four parts in Section \ref{sectFDCTbounds}. Theorem \ref{thmFDCT} is in fact a  corollary of Theorem \ref{thmFDCTi}. By the correspondence in \cite{Liu},  we obtain sharp lower bounds of $\delta_i(f)$.  As an immediate application, we prove Theorem \ref{thmBogo} at the end of  \S \ref{sectFDCTbounds}. In \S \ref{sectMaing=2}, we prove results of lower bounds of modular invariants of fibrations of genus 2. We prove Theorem \ref{thmlambdakappa} first, and then prove  the rigidity property (Theorem \ref{thmrigidity}),  using the theory of Chern numbers of fibers (see \cite{Ta10}). Theorem \ref{thmexistence},  the optimum of the lower bounds in Theorem \ref{thmlambdakappa}, is proved in \S \ref{sectionexistence}. We prove results (Theorem \ref{thmg3mod}, Theorem \ref{thmg=3hyp} and Theorem \ref{thmg4mod}) for fibrations of genus $g\geq3$ in \S\ref{sectMaing3}.

\section{Preliminaries}\label{sectprel}

\subsection{Modular invariants}
Let $\Delta_0,\Delta_1,\ldots,\Delta_{[g/2]}$ be the boundary divisors of $\overline{\CM}_g$,  $\delta_i$ be the $\mathbb{Q}$-divisor classes corresponding to $\Delta_i$, and $\delta_i(f)$ be the modular invariants corresponding to $\delta_i$. Then
\begin{equation}
  \delta(f)=\delta_0(f)+\delta_1(f)+\cdots+\delta_{[g/2]}(f).
\end{equation}

 If $g=1$, then $\delta(f)$ is the number of poles of the $J$-function of the family (see \cite{Li16} for generalization). When $g\geq2$, it is shown (\cite{Ta94,Ta96}) that $\lambda(f)=0$ if and only if $\kappa(f)=0$ if and only if $f$ is an isotrivial family. In this paper, we always assume that $f$ is non-isotrivial, then $\lambda(f)$ and $\kappa(f)$ are positive rational number.

Let $F$ be a singular fiber of $f$, and $\tilde F$ be its $d$-th semistable model (\cite[p.207]{LT17}). Let $p$ be a node of $\tilde F$. We say $p$ is of {\it type 0} if the normalization of $\tilde F$ at $p$ is connected. Otherwise, the normalization at $p$ has two connected components, and we say $p$ is of {\it type $i$}, where $i$ is the minimum of the arithmetic genera of the two components.  Denote by $\delta_i(\tilde F)$ the number of nodes of type $i$ in $\tilde F$, then we define
\begin{equation}\label{edeltaF}
\delta_i(F):=\frac{\delta_i(\tilde F)}d,~~~(i=0,1,\ldots,[g/2]),
\end{equation}
 which is independent of the choice of the semistable model $\tilde F$ of $F$.  Let $F_1,\ldots,F_s$ be all singular fibers of $f$. It is shown that, in \cite{LT17},
\begin{equation}\label{edeltaf}
\delta_i(f)=\delta_i(F_1)+\cdots+\delta_i(F_s),~~~i=0,1,\ldots,[g/2].
\end{equation}

 If we restrict $f$ to a neighborhood of $f(F)\in C$, we can get a local family $f_F$ whose dual graph is $G(F)$, and we denote by $\phi_F$ the pseudo-periodic map determined by $f_F$. On the other hand,  let $\phi\in\mathrm{Mod}(\Sigma_g)$ be a pseudo-periodic map of negative twist. Then, for each $i\geq0$, we have  (\cite[Theorem 1.2]{Liu})
 \begin{equation}\label{eqdeltai}
   \delta_i(F_\phi)=\delta_i(f_\phi)=\sum_{\gamma\in\SC_i}|c(\phi,\gamma)|,
 \end{equation}
 where $\SC_i=\{\gamma\in\SC: \gamma~\mbox{is~of~type~}i\}$. So, if $\delta(F_\phi)=0$, then $F_\phi$ has smooth  reduction, and $\phi$ has periodic monodromy.

\subsection{Valencies of  periodic maps}\label{subsectValency}
Let $\Sigma$ be a connected real 2-dimensional
manifold with or without boundary. When we emphasize its complex
structure, we call $\Sigma$ a Riemann surface.

Let $\phi:\Sigma\to \Sigma $ be a periodic  homeomorphism  of order $n\geq2$, and $p$ be a point on
$\Sigma$. There is a positive integer $m_p$ such that the
points $p,\phi(p),\ldots,\phi^{m_p-1}(p)$ are mutually distinct and
$\phi^{m_p}(p)=p$. If $m_p=n$, we call the point $p$ a {\it
simple point} of $\phi$, while if $m_p<n$, we call $p$ a {\it
multiple point} of $\phi$.

 Let $\gamma$ be a cut curve in $\SC$ and  $m=m_{\vec{\gamma}}$ be the smallest positive integer such that
$\phi^m(\vec{\gamma})=\vec \gamma$ (i.e., $\phi^{m}(\gamma)=\gamma$ as a set, and $\phi^m$
preserving the orientation of $\gamma$). The restriction of $\phi^m$ to $\vec
\gamma$ is a periodic map of order, say, $\lambda\geq1$.  Let $q$ be any point on  $\gamma$, and suppose that the
images of $q$ under the iteration of $\phi^m$ are ordered
$(q,\phi^{m\sigma}(q),\ldots,
\phi^{(\lambda-1)m\sigma}(q))$ viewed in the direction of $\vec \gamma$,
where $\sigma$ is an integer with $0\leq \sigma \leq \lambda-1$,
$\mathrm{gcd}(\sigma,\lambda)=1$, and $\sigma=0$ iff $\lambda=1$.
 The triple $(m, \lambda,\sigma)$ is called
the {\it valency} of $\vec \gamma$ with respect to $\phi$.

We define the {\it valency of a boundary curve} (i.e., a
connected component of the boundary $\partial \Sigma$) as its
valency with respect to $\phi$, assuming it has the orientation induced
by the surface $\Sigma$. The {\it valency of a multiple point} $p$
is defined to be the valency of the boundary curve $\partial D_p$, oriented from the outside of a disk neighborhood $D_p$ of $p$.

  Let $\Sigma$ be a surface of genus $g$ with $k$ boundary curves $\partial_1,\ldots,\partial_k$. Let $\phi:\Sigma\to\Sigma$ be an orientation-preserving homeomorphism which satisfies:

   (1) there is a disjoint union of simple closed curves $\SC=\amalg_{j=1}^r \gamma_j$ such that $\SC$ and $\partial\Sigma=\amalg_{j=1}^k\partial_j$ do not intersect each other,

   (2) $\Sigma-\SC$ is connected,

   (3) $\phi(\SC)=\SC$ and $\phi|_{\Sigma-\SC}$ is periodic.

  \noindent Then we can  extend $\phi$ to a periodic map on a closed surface $\tilde \Sigma$ easily (\cite[Lemma 1.2]{AI02}), and classify the valency data for each $(g,r,k)$ similarly to Lemma \ref{lemmaValencies} (see \cite[Section 2.2]{AI02}).

Suppose $\Pi:\Sigma\to \Sigma'$ is the $n$-fold
cyclic covering induced by $\phi$, where $\Sigma'$ is the quotient surface of $\Sigma$ with respect to $\phi$. Let $\{q_1,\ldots,q_l\}\subseteq \Sigma'$ be the
set of branch points. If $\tilde q_i$ is a point of the pre-image $\Pi^{-1}(q_i)$
of $q_i$, and let the valency of $\tilde q_i$ be $(m_i,\lambda_i,\sigma_i)$.
Then we know that $m_i$ is the number of points in $\Pi^{-1}(q_i)$ and $\lambda_i=n/m_i$. Since the valencies of points in $\Pi^{-1}(q_i)$ are the same,  we can define the valency of $q_i$ to be the valency of $\tilde q_i$.

For brevity's sake, if we have the data of valencies $(n/\lambda_i,\lambda_i,\sigma_i)~(1\leq i\leq l)$, we symbolically write ~$\sigma_1/\lambda_1+\cdots+\sigma_l/\lambda_l$ which is called the {\it total valency}. We also write the order $n$ of the map and the genus $g'$ of $\Sigma'$. However if $g'=0$, the genus is omitted. A periodic map can be represented by its total valency.  For the reader's convenience, we list the classification of periodic maps in \cite[Lemma 1.4]{AI02} here.

\begin{lemma}\label{lemmaValencies}
Non-identical conjugacy classes of periodic maps of closed surfaces of genus $1\leq g\leq 2$ are classified as follows:\\
{\upshape(i)}~$g=1$\\
\indent$(1)$~ $n=6; ~1/6+1/3+1/2,5/6+2/3+1/2.$\\
\indent$(2)$~ $n=4; ~1/4+1/4+1/2,3/4+3/4+1/2.$\\
\indent$(3)$~ $n=3; ~1/3+1/3+1/3,2/3+2/3+2/3.$\\
\indent$(4)$~ $n=2; ~1/2+1/2+1/2+1/2.$\\
\indent$(5)$~ $g'=1,~n$ is arbitrary and $\Pi:\Sigma\to\Sigma'$ is an unramified covering.\\
{\upshape(ii)}~$g=2$\\
\indent$(1)$~ $n=10; ~1/10+2/5+1/2,3/10+1/5+1/2,7/10+4/5+1/2,9/10+3/5+1/2.$\\
\indent$(2)$~ $n=8; ~1/8+3/8+1/2,5/8+7/8+1/2.$\\
\indent$(3)$~ $n=6; ~1/6+1/6+2/3,5/6+5/6+1/3,1/3+2/3+1/2+1/2.$\\
\indent$(4)$~ $n=5; ~1/5+1/5+3/5,1/5+2/5+2/5,2/5+4/5+4/5,3/5+3/5+4/5.$\\
\indent$(5)$~ $n=4; ~1/4+3/4+1/2+1/2.$\\
\indent$(6)$~ $n=3; ~1/3+1/3+2/3+2/3.$\\
\indent$(8)$~ $g'=1,~n=2, 1/2+1/2.$
\end{lemma}

\subsection{Representation of a pseudo-periodic map}\label{sectperiod}
Let $\phi:\Sigma_g\to\Sigma_g$ be a pseudo-periodic map, and $\SC=\{\gamma_1,\ldots,\gamma_r\}$ be the admissible system of cut curves. Then the restriction of $\phi$ on $\SB=\Sigma_g-\SC$ is isotopic to a periodic map.

Now we use a weighted graph $G_\SC$ to denote the decomposition $\Sigma=\SB\bigcup\SC$. A vertex $v$ in $G_\SC$ corresponds to a connected component $B_v$ of $\SB$, and an edge $e$ corresponds to a separated cut curve $\gamma_e$ in $\SC$, where $\gamma_e$ is adjacent to two connected components of $\SB$. We define the weight of a vertex $v$ to be $g(B_v)+\rho(v)$, where $g(B_v)$ is the genus of $B_v$, and $\rho(v)$ is the number of cut curves only adjacent to $B_v$.We use a small circle to denote a vertex, and the number inside the small circle  means $g(B_v)+\rho(v)$. We omit the number when it is zero.

  Note that a weighted graph may represent different decompositions. For example,
 the graph (II) in Lemma \ref{lemclassificationg=2} represents four types of decompositions, that is, the component corresponding to $v_1$ has genus $i_1$ and is adjacent to $1-i_1$ non-separating curves in $\SC$, and the component corresponding to $v_2$ has genus $i_2$ and is adjacent to $1-i_2$ non-separating curves in $\SC$ ($0\leq i_1\leq 1,~0\leq i_2\leq 1$).

The map $\phi:\Sigma_g=\SB\bigcup\SC\to \Sigma_g=\SB\bigcup\SC$ induces an automorphism $\sigma_\phi$ on the weighted graph $G_\SC$. Here an {\it automorphism}  of $G_\SC$ means an automorphism of the graph such that the weight $(g(B_v),\rho(v))$ coincides with $(g(B_{\sigma(v)}),\rho(\sigma(v)))$ for each vertex $v$ of $G_\SC$, see \cite[Section 3.3]{AI02} for an example.

For each cut curve $\gamma\in\SC$, there exists a minimal positive integer
$\alpha$ such that $\phi^{\alpha}(\vec{\gamma})=\vec{\gamma}$.  The curve $\gamma$ is said to be {\it amphidrome} if $\alpha$ is even and
$\phi^{\alpha/2}(\vec{\gamma})=-\vec{\gamma}$, (where $\vec\gamma$ and $-\vec\gamma$ denote the same $\gamma$ with the opposite directions assigned) and {\it non-amphidrome} otherwise. There exists a minimal positive integer $L$ such that the restriction of $\phi^{L}$ to an annulus of ${\gamma}$ is isotopic to a Dehn
twist of $e$ times $(e\in \mZ)$.   The rational number $e\alpha/L$ is called the {\it screw number} of $\phi$ about $\gamma$, and is denoted by $s(\gamma)$ (see \cite{Ni44}).
We may always assume that $s(\gamma)\neq0$ for each $\gamma\in\SC$ (see \cite[p.5]{MM11}).

For each $\gamma\in\SC$, denote by $m_\gamma$ the length of the cyclic orbit of $\gamma$ under the permutation caused by $\phi$, that is,
$$m_\gamma=\#\{\phi^k(\gamma):k\in \mathbb{N}\}.$$

Our classification is based on the following theorem in \cite{MM11}.
\begin{theorem}\label{thmMMphi}
  The conjugacy class of a pseudo-periodic map $\phi:\Sigma_g\to\Sigma_g$ of negative twist is determined by the following data: an admissible system $\SC$  of cut curves, the induced automorphism $\sigma_\phi$ of $G_\SC$, the screw numbers $s(\gamma)$ for each $\gamma\in\SC$, and the valency data of the periodic maps which stabilize the connected components of $\Sigma_g-\SC$.
\end{theorem}

 Now we give the formula of fractional Dehn twist coefficients in \cite[Theorem 4.5]{Liu} and the formula of screw number
 (see \cite[Section 2.1]{AI02}). Let $\SA_\gamma$ be the annular neighborhood of $\gamma$. Let $(m_1,\lambda_1,\sigma_1)$ and $(m_2,\lambda_2,\sigma_2)$ be the valencies of the two boundary curves of $\SA_\gamma$. If $\gamma$ is non-amphidrome,  then $m_1=m_2=m_\gamma$, and
\begin{equation}\label{eqFDCTnon-a}
   |c(\phi,\gamma)|=\frac{|s(\gamma)|}{m_\gamma}
=\frac1{m_\gamma}(\frac{\mu_1}{\lambda_1}+\frac{\mu_2}{\lambda_2}+K),
\end{equation}
where $K\geq-1$ is an integer, and $\mu_i$ are integers with $$\sigma_i\mu_i\equiv1\mod \lambda_i, ~~~~ 0\leq\mu_i\leq\lambda_i-1.$$
If $\gamma$ is amphidrome, then the two boundary curves have the same valency $(2m_\gamma,\lambda,\sigma)$ where $2m_\gamma=\alpha$, and
   \begin{equation}\label{eqFDCT-a}
     |c(\phi,\gamma)|=\frac{|s(\gamma)|}{2m_\gamma}
=\frac1{m_\gamma}(\frac{\mu}{\lambda}+K),
\end{equation}
where $K\geq0$ is an integer, and $\mu$ is an integer with $$\sigma\mu\equiv1\mod \lambda, ~~~~0\leq\mu\leq\lambda-1.$$

     Let  $\gamma$  be a cut curve adjacent to connected components $B_{v_1}$, $B_{v_2}$ with $B_{v_1}\neq B_{v_2}$.  Suppose that $\gamma_k:=\phi^k(\gamma)$  ($k=1,\ldots,{m_\gamma}$) are all adjacent to $B_{v_i}~(i=1,2)$.  Denote the valencies of the two boundary curves  of $\SA_\gamma$ by $(m_1,\lambda_{v_1,\gamma},\sigma_{v_1,\gamma})$ and $(m_2,\lambda_{v_2,\gamma},\sigma_{v_2,\gamma})$ respectively.
      If $a$ is the minimal positive integer such that $\phi^a(B_{v_i})=B_{v_i}$, then the restriction $\phi^a|_{B_{v_i}}$ of $\phi^a$ to $B_{v_i}$ is periodic. Let $n(B_{v_i})$ be the order of $\phi^a|_{B_{v_i}}$.

      If $\phi$ interchanges $B_{v_i}$ $(i=1,2)$, then $\phi(\vec{\gamma}_k)=-\vec{\gamma}_{k+1}~(k=1,\ldots,m_\gamma)$, where $\gamma_{m_\gamma+1}:=\gamma_1=\gamma$.  So $\phi^{m_\gamma}(\vec{\gamma})=(-1)^{m_\gamma}\vec{\gamma}$. Moreover, if $m_\gamma$ is even, then $\gamma$ is non-amphidrome and $m_{\vec{\gamma}}=m_\gamma$; if $m_\gamma$ is odd, then $\gamma$ is amphidrome and $m_{\vec{\gamma}}=2m_\gamma$. Since $\phi$ interchanges $B_{v_i}$,  $n(B_{v_i})=m_{\vec{\gamma}}\lambda_{v_i,\gamma}/2$.
     Moreover, if $\phi$ does not interchange $B_{v_i}$, then $\gamma$ is non-amphidrome.  Thus we have that, for $B_{v_1}\neq B_{v_2}$,
\begin{align}\label{lami}
 \lambda_{v_i,\gamma}=
   \begin{cases}
 \frac{n(B_{v_i})}{m_\gamma}, & \mbox{if  $\phi$ does not interchange $B_{v_i}$}; \\
\frac{n(B_{v_i})}{m_\gamma}, & \mbox{if $\phi$ interchanges $B_{v_i}$ and $\gamma$ is amphidrome}; \\
 \frac{2n(B_{v_i})}{ m_\gamma}, & \mbox{if $\phi$ interchanges $B_{v_i}$ and $\gamma$ is non-amphidrome}.
\end{cases}
\end{align}
In particular, if $m_\gamma$ is odd, then $\lambda_{v_i,\gamma}=\frac{n(B_{v_i})}{m_\gamma}$.

\section{Bounds of fractional Dehn twist coefficients}\label{sectFDCTbounds}

\subsection{Genus 2 case}\label{subsectFDCTg=2}
First we give the classification of  decompositions of Riemann surfaces of genus two.
\begin{lemma}\label{lemclassificationg=2}
  The decompositions of a Riemann surface of genus two by an admissible system of cut curves can be classified in terms of weighted graphs (I)-(III) as follows.
\end{lemma}
  
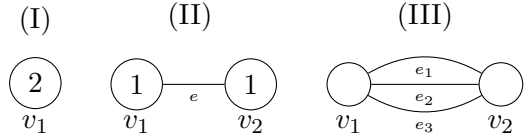
\begin{center}
   \begin{tikzpicture}
\usetikzlibrary{positioning}
  \node [shape=circle,draw] (v1) at (0,0) {2};
  \node [above]  at (0,0.5) {(I)};  \node [below]  at (0,-0.3) {$v_1$};
\end{tikzpicture}
~~~
\begin{tikzpicture}
\usetikzlibrary{positioning}
  \node [shape=circle,draw] (v1) at (0,0) {1};
  \node [shape=circle,draw] (v2) at (1.5,0) {1};
  \draw (v1) node[below=8pt]{$v_1$}--node[below]{\tiny{$e$}} (v2)node[below=8pt]{$v_2$};
  \node at (0.7,0.9){(II)};
\end{tikzpicture}
~~~
\begin{tikzpicture}
\usetikzlibrary{positioning}
  \node [shape=circle,draw] (v1) at (0,0) {$\mbox{~~}$};
  \node [shape=circle,draw] (v2) at (2,0) {$\mbox{~~}$ };
      \node [below=8pt] at (v2)  {$v_2$};
          \node [below=8pt] at (v1)  {$v_1$};
  \draw  (v1)to [bend left=30] node [below]{\tiny{${e_1}$}} (v2);
  \draw  (v1)to [bend right=30] node [below]{\tiny{${e_3}$}} (v2);
  \draw (v1) to node [below]{\tiny{$ {e_2}$}} (v2);
  \node at (1,0.9){(III)};
\end{tikzpicture}
  \end{center}
  \begin{proof}
    This problem is equivalent to classifying stable curves of genus two, which is trivial.
  \end{proof}

\begin{theorem}\label{thmType1-g=2}
  Let $\phi\in\mathrm{Mod}({\Sigma_2})$ be a pseudo-periodic map of negative twist, and $\gamma$ be a cut curve of type 1, then
  $$|c(\phi,\gamma)|\geq\frac1{12},$$
  and the equality holds if and only if $(G(F_\phi),\gamma)$ is either Figure (2-1a) or Figure (2-1b).
\end{theorem}
\begin{proof}
The possible cut curve $\gamma$ is $e$ in (II), and $m_\gamma=1$.

\noindent({\bf II1}). Assume that $\phi$ does not interchange $B_{v_i}$. Then $\gamma$ is non-amphidrome.

Claim: if $\rho(v_i)=1,g(B_{v_i})=0$ for some $i$, then $|c(\phi,\gamma)|>\frac1{12}$.

Proof of Claim: We may assume that $\rho(v_1)=1, g(B_{v_1})=0$.  Let $\gamma'$ be the cut curve only adjacent to $B_{v_1}$, and $\SA_{\gamma'}$ be the annular neighbourhood of $\gamma'$. Then $\phi$ maps boundary curves of $\SA_{\gamma'}$ (resp. $\SA_{\gamma}$) to boundary curves of $\SA_{\gamma'}$ (resp. $\SA_{\gamma}$). Thus $n(B_{v_1})\leq2$,  since the automorphism of Riemann sphere with three fixed points is identity. Hence $\lambda_{v_1,\gamma}=n(B_{v_1})\leq2$. Now consider $B_{v_2}$, we have that $n(B_{v_2})\leq6$: if $\rho(v_2)=0$, then $g(B_{v_2})=1$ and $n(B_{v_2})\leq6$  by Lemma \ref{lemmaValencies} (i); if $\rho(v_2)=1$, then $g(B_{v_2})=0$ and $n(B_{v_2})\leq2$ as above. Thus, by (\ref{eqFDCTnon-a}),
 $$|c(\phi,\gamma)|=\frac{\mu_{v_1,\gamma}}{\lambda_{v_1,\gamma}}
  +\frac{\mu_{v_2,\gamma}}{\lambda_{v_2,\gamma}}+K\geq\frac1{\mathrm{lcm}(\lambda_{v_1,\gamma},\lambda_{v_2,\gamma})}>\frac1{12},$$
  and we finish the proof of Claim.

Now we only need to prove the case that $\rho(v_i)=0,g(B_{v_i})=1,~i=1,2$. Since there is one edge $e$ adjacent to $v_i~(i=1,2)$, we know that the restriction of $\phi$  on $B_{v_i}$ can not induce an unramified covering of degree $n\geq 2$. Thus $n(B_{v_i})\leq 6$ and the valency data are classified in Lemma \ref{lemmaValencies} (i). So, by (\ref{eqFDCTnon-a}),
  $$|c(\phi,\gamma)| =\frac{\mu_{v_1,\gamma}}{\lambda_{v_1,\gamma}}
  +\frac{\mu_{v_2,\gamma}}{\lambda_{v_2,\gamma}}+K \geq  \frac1{12}.$$
If the equality holds, then $K=-1$. Furthermore, by  Lemma \ref{lemmaValencies} and Theorem \ref{thmMMphi}, the cut curves and the pseudo-periodic maps with the lowest bound are classified  as follows:
\vspace{2mm}

(2-1a)~ $B_{v_1}: {\bf\frac13}+\frac13+\frac13 , ~~B_{v_2}: {\bf\frac34}+\frac34+\frac12$, $K=-1$;  \vspace{1mm}

(2-1b)~ $B_{v_1}: {\bf\frac5{6}}+\frac23+\frac12, ~~B_{v_2}: {\bf\frac14}+\frac14+\frac12$, $K=-1$.
\vspace{2mm}

\noindent Here we write valency data of  $\gamma$ by bold face characters. For the reader's convenience,  we take the case (2-1b) as an example, the valencies of the two boundary curves of $\SA_\gamma$ are $(m_{v_1,\gamma},\lambda_{v_1,\gamma},\sigma_{v_1,\gamma})=(1,6,5)$ and $(m_{v_2,\gamma},\lambda_{v_2,\gamma},\sigma_{v_2,\gamma})=(1,4,1)$. Thus $\mu_{v_1,\gamma}=5,\mu_{v_2,\gamma}=1$, and
$$|c(\phi,\gamma)|=\frac56+\frac14+(-1)=\frac1{12}.$$

 Using the correspondence in \cite[Section 4]{Liu}, it is easy to check that the dual graphs of the above two pseudo-periodic maps are Figure (2-1a) and Figure (2-1b) respectively.
 In the following, we always use same labels (for example, (2-1a)) for pseudo-periodic maps and their corresponding dual graphs without mention.

 \noindent({\bf II2}). Assume that $\phi$ interchanges $B_{v_i}$. Then  $\gamma$ is amphidrome and $\lambda_{v_i,\gamma}=n(B_{v_i})$ $(i=1,2)$ by (\ref{lami}). As above, we have that   $n(B_{v_i})\leq6$.  So, by (\ref{eqFDCT-a}),
 $$|c(\phi,\gamma)|=\frac1{m_\gamma}(\frac{\mu_{v_i,\gamma}}{\lambda_{v_i,\gamma}}
  +K)\geq\frac1{n(B_{v_i})}\geq\frac1{6}.$$

By Theorem \ref{thmMMphi}, there is no other pseudo-periodic maps $(\phi,\gamma)$ with $|c(\phi,\gamma)|=\frac1{12}$.
\end{proof}

\begin{corollary}\label{coro2-1}
  Let $f$ be a family of curves of genus $2$ with $\delta_1(f)\neq0$, then
  $$\delta_1(f)\geq\frac1{12},$$
  and the equality holds if and only if  all the singular fibers of $f$ have smooth reduction except one whose dual graph is either Figure (2-1a) or Figure (2-1b).
\end{corollary}
\begin{proof}
Let $F$ be a singular  fiber of $f$, then

 Claim (*): $\delta_1(F)\geq\frac1{12}$, and $\delta_1(F)=\frac1{12}$ if and only if the dual graph of $F$ is either Figure (2-1a) or Figure (2-1b).

Proof of Claim (*):  By (\ref{eqdeltai}) and Theorem \ref{thmType1-g=2}, we have
  $$\delta_1(F)=\sum_{\gamma\in\SC_{\phi_{F},1}}|c(\phi_{F},\gamma)|\geq\frac1{12}.$$
  If $\delta_1(F)=\frac1{12}$, then $\phi_F$ has only one cut curve $\gamma$ of type 1, and the possible dual graphs of  $(\phi_F,\gamma)$ are Figure (2-1a) and Figure (2-1b). Hence we obtain the claim.

 Since $\delta_1(f)\neq0$,  there is a singular fiber of $f$, say $F_1$, with $\delta_1(F_1)\neq0$ by (\ref{edeltaf}).    So
 $$\delta_1(f)\geq\delta_1(F_1)\geq\frac1{12},$$
 and $\delta_1(f)=\delta_1(F_1)=\frac1{12}$ if and only if $\delta_1(F_1)=\frac1{12}$ and $\delta_1(F_i)=0,~i\geq2$. Then we complete the proof by Claim (*).
\end{proof}

\begin{theorem}\label{thmType0-g=2}
  Let $\phi\in\mathrm{Mod}({\Sigma_2})$ be a pseudo-periodic map of negative twist, and $\gamma$ be a cut curve of type 0, then
  $$|c(\phi,\gamma)|\geq\frac1{3},$$
  and the equality holds if and only if $(G(F_\phi),\gamma)$ is one of  figures: (2-0a) -- (2-0d).
\end{theorem}
\begin{proof}
The possible weighted graphs $G_\SC$ are (I), (II) and (III) classified in Lemma \ref{lemclassificationg=2}.

 \underline{ Case (I)}
  In this case,  $\gamma$ is adjacent to $B_{v_1}$ with $g(B_{v_1})+\rho({v_1})=2$ and $\rho({v_1})\geq1$.  Let $(m_1,\lambda_1,\sigma_1)$ and $(m_2,\lambda_2,\sigma_2)$ be the valencies of the boundary curves of $\SA_\gamma$.

\noindent({\bf{I1}}). First consider $\rho({v_1})=1$. Then $g(B_{v_1})=1$ and $m_\gamma=1$.

\noindent If $\gamma$ is amphidrome, then  $\lambda=n(B_{v_1})/2\leq 3$. Thus, by (\ref{eqFDCT-a})  and Lemma \ref{lemmaValencies},
    $$|c(\phi,\gamma)| =\frac{\mu}{\lambda}+K \geq\frac1{3}.$$
\noindent Furthermore, the equality holds if and only if $K=0$ and the valency data of the boundary curves of $\SA_\gamma$ is

\vspace{2mm}
(2-0a) $B_{v_1}: {\frac1{6}}+{\bf\frac 13}+{\frac12}$.
\vspace{1mm}

\noindent    If $\gamma$ is non-amphidrome,  then $\lambda_1=\lambda_2=n(B_{v_1})$. By  (\ref{eqFDCTnon-a}), we have that
   $$|c(\phi,\gamma)| =\frac{\mu_1}{\lambda_1}+\frac{\mu_2}{\lambda_2}+K \geq\frac13.$$
   Furthermore, the equality holds if and only if $K=-1$ and the valency data of the boundary curves of $\SA_\gamma$ is

\vspace{2mm}
(2-0b) $B_{v_1}: {\bf\frac23}+{\bf\frac23}+\frac 23$.
\vspace{1mm}

\noindent({\bf{I2}}). Now we consider $\rho({v_1})=2$, then $g(B_{v_1})=0$. Let the two cut curves be $\gamma=\gamma_1$ and $\gamma_2$. Then $m_\gamma\leq2$.
 Note that $\phi|_{B_{v_1}}$ maps $\SA_{\gamma_i}$ to $\SA_{\gamma_j}$, and  maps boundary curves of $\SA_{\gamma_i}$ to those of $\SA_{\gamma_j}$, where  $1\leq i,j\leq 2$.  If $m_\gamma=1$, then $\lambda\leq n(B_{v_1})\leq2$, and thus $|c(\phi,\gamma)|\geq\frac1{2}.$ If $m_\gamma=2$, then  $\lambda=1$ which is independent of the action of $\phi$ on boundary curves (note that if $\gamma$ is amphidrome, then $\lambda=n(B_{v_1})/2m_\gamma$). Thus $|c(\phi,\gamma)|\geq\frac12$, by (\ref{eqFDCTnon-a}) and (\ref{eqFDCT-a}).

\underline{Case (II)} In this case, we may assume that $\gamma$ is adjacent to $B_{v_1}$ only and $m_\gamma\leq2$, then we have $\rho(v_1)=1, g(B_{v_1})=0$. We have that $\lambda=1$ which is independent of the action of $\phi$ on boundary curves, and thus $|c(\phi,\gamma)|\geq\frac12$.

\underline{Case (III)} We may assume that $\gamma=e_1$.

\noindent({\bf{III1}}). Assume that $m_{\gamma}=3$.   If $\phi$ does not interchange $B_{v_i}$, then $\gamma$ is non-amphidrome, $n(B_{v_i})=3$ and $\lambda_{v_i,\gamma}=n(B_{v_i})/m_{\gamma}=1$ for $i=1,2$. So
 $|c(\phi,\gamma)|\geq\frac13$ by (\ref{eqFDCTnon-a}). Furthermore, the equality holds if and only if $K=1$ and the valencies of the boundary curves of $\CA_\gamma$ are

 \vspace{2mm}
(2-0c) $B_{v_1}: {\frac1{3}}+{\frac 23}+{\bf1}, ~~B_{v_2}: {\frac1{3}}+{\frac 23}+{\bf1}$.
\vspace{1mm}

\noindent If $\phi$ interchanges $B_{v_i}$, then $\gamma$ is amphidrome,  $n(B_{v_i})=3$ and $\lambda_{v_i,\gamma}=n(B_{v_i})/m_{\gamma}=1$, see (\ref{lami}). So
 $|c(\phi,\gamma)|\geq\frac13$ by (\ref{eqFDCT-a}). Furthermore, the equality holds if and only if $K=1$ and the valency data  of the boundary curves of $\CA_\gamma$ is

 \vspace{2mm}
(2-0d) $B_{v_1}: {\frac1{3}}+{\frac 23}+{\bf1}$.
\vspace{1mm}

\noindent({\bf{III2}}). Assume that $m_{\gamma}=2$.  If $\phi$ does not interchange $B_{v_i}$, then $\gamma$ is non-amphidrome, $n(B_{v_i})=2$ and $\lambda_{v_i,\gamma}=1$; if $\phi$ interchanges $B_{v_i}$,  then $\gamma$ is also non-amphidrome, $n(B_{v_i})=1$ and $\lambda_{v_i,\gamma}=1$ by (\ref{lami}). Thus, we always have that $\lambda_{v_i,\gamma}=1$ and  $|c(\phi,\gamma)|\geq\frac12.$

 \noindent({\bf{III3}}). Assume that $m_{\gamma}=1$. If  $\phi$ does not interchange $B_{v_i}$, then $\gamma$ is non-amphidrome, $\lambda_{v_i,\gamma}=n(B_{v_i})\leq2$; if $\phi$ interchanges $B_{v_i}$, then $\gamma$ is amphidrome and $n(B_{v_i})\leq2$. So $\lambda_{v_i,\gamma}=n(B_{v_i})\leq2$ by (\ref{lami}). Thus
 $|c(\phi,\gamma)|\geq\frac12$.

\end{proof}

\begin{corollary}\label{coro2-0}
 Let $f$ be a fibration of genus $2$ with $\delta_0(f)\neq0$, then
  $$\delta_0(f)\geq\frac1{3},$$
  and the equality holds if and only if  all the singular fibers of $f$ have smooth reduction except one whose dual graph is either Figure (2-0a) or Figure (2-0b).
\end{corollary}
\begin{proof}
Similar to the proof of Corollary \ref{coro2-1}.
\end{proof}

\subsection{Genus 3 case}
In this subsection, we use the same method as Section \ref{subsectFDCTg=2} to discuss lower bounds of fractional Dehn twist coefficients in genus 3 case.

\begin{lemma}
  The decompositions of a Riemann surface of genus three by an admissible system of cut curves can be classified in terms of weighted graphs (A)-(O) as follows.

{\upshape
\begin{tikzpicture}
\usetikzlibrary{positioning}
  \node [shape=circle,draw] (v1) at (0,0) {3};
  \node [above]  at (0,0.5) {(A)};  \node [below]  at (0,-0.3) {$v_1$};
\end{tikzpicture}
~~~
\begin{tikzpicture}
\usetikzlibrary{positioning}
  \node [shape=circle,draw] (v1) at (0,0) {2};
  \node [shape=circle,draw] (v2) at (1.5,0) {1};
  \draw (v1) node[below=8pt]{$v_1$}--node[below]{\tiny{$e_1$}} (v2)node[below=8pt]{$v_2$};
  \node at (0.7,0.9){(B)};
\end{tikzpicture}
~~~
\begin{tikzpicture}
\usetikzlibrary{positioning}
  \node [shape=circle,draw] (v1) at (0,0) {1};
  \node [shape=circle,draw] (v2) at (1.5,0) {1};
  \node [shape=circle,draw] (v3) at (3,0) {1};
  \draw (v1)node[below=8pt]{$v_1$}--node[below]{\tiny{$e_1$}} (v2)node[below=8pt]{$v_3$}--node[below]{\tiny{$e_2$}} (v3)node[below=8pt]{$v_2$};
  \node at (1.5,0.9){(C)};
\end{tikzpicture}
~~~
\begin{tikzpicture}
\usetikzlibrary{positioning}
  \node [shape=circle,draw] (v1) at (0,0) {1};
  \node [shape=circle,draw] (v4) at (1.5,0){$\mbox{~~}$};
  \node [shape=circle,draw] (v3) at (3,0) {1};
    \node [shape=circle,draw] (v2) at (1.5,1.2) {1};
  \draw (v1)node[below=8pt]{$v_1$}-- node[above]{\tiny{$e_1$}}(v4) --node[above]{\tiny{$e_3$}} (v3) node[below=8pt]{$v_3$};\draw (v2) node[right=6pt]{$v_2$}--node[left]{\tiny{$e_2$}} (v4) node[below=8pt]{$v_4$};
  \node at (1.5,2.1){(D)};
\end{tikzpicture}\\

\begin{tikzpicture}
\usetikzlibrary{positioning}
  \node [shape=circle,draw] (v1) at (0,0) {1};
  \node [shape=circle,draw] (v2) at (2,0) {1};
  \node [below=8pt] at (v2)  {$v_2$};
  \draw  (v1)node[below=8pt]{$v_1$} to [bend left=15] node [above]{\tiny{$e_1$}}(v2);
  \draw  (v1) to [bend right=15]node [below]{\tiny{$e_2$}}(v2);
  \node at (1,0.9){(E)};
\end{tikzpicture}
~~~
\begin{tikzpicture}
\usetikzlibrary{positioning}
  \node [shape=circle,draw] (v1) at (0,0) {1};
  \node [shape=circle,draw] (v3) at (1.5,0) {$\mbox{~~}$};
  \node [shape=circle,draw] (v4) at (3,0) {$\mbox{~~}$};
  \node [shape=circle,draw] (v2) at (4.5,0) {1};
    \node [below=8pt] at (v4)  {$v_4$};
  \draw  (v3) node[below=8pt]{$v_3$} to [bend left=15] node [above]{\tiny{$e_1$}}(v4);
  \draw  (v3)to [bend right=15] node [below]{\tiny{$e_2$}} (v4);
  \draw (v1) node[below=8pt]{$v_1$} to node[below]{\tiny{$e_3$}} (v3);\draw (v4)to node[below]{\tiny{$e_4$}} (v2) node[below=8pt]{$v_2$};
  \node at (2.2,0.9){(F)};
\end{tikzpicture}\\

\begin{tikzpicture}
\usetikzlibrary{positioning}
  \node [shape=circle,draw] (v1) at (0,0) {1};
  \node [shape=circle,draw] (v2) at (1.5,0) {$\mbox{~~}$ };
  \node [shape=circle,draw] (v3) at (3,0) {1};
    \node [below=8pt] at (v2)  {$v_2$};
  \draw  (v1) node[below=8pt]{$v_1$} to [bend left=15] node [above]{\tiny{$e_1$}} (v2);
  \draw  (v1)to [bend right=15] node [below]{\tiny{$e_2$}}  (v2);
  \draw (v2)to node[below]{\tiny{$e_3$}} (v3) node[below=8pt]{$v_3$};
  \node at (1.5,0.9){(G)};
\end{tikzpicture}
~~~
\begin{tikzpicture}
\usetikzlibrary{positioning}
  \node [shape=circle,draw] (v1) at (0,0) {1};
  \node [shape=circle,draw] (v2) at (2,0) {$\mbox{~~}$ };
      \node [below=8pt] at (v2)  {$v_2$};
          \node [below=8pt] at (v1)  {$v_1$};
  \draw  (v1)to [bend left=30] node [below]{\tiny{${e_1}$}} (v2);
  \draw  (v1)to [bend right=30] node [below]{\tiny{${e_3}$}} (v2);
  \draw (v1) to node [below]{\tiny{$ {e_2}$}} (v2);
  \node at (1,0.9){(H)};
\end{tikzpicture}
~~~
\begin{tikzpicture}
\usetikzlibrary{positioning}
  \node [shape=circle,draw] (v1) at (0,0) { $\mbox{~~}$};
  \node [shape=circle,draw] (v2) at (1.5,0) { $\mbox{~~}$};
  \node [shape=circle,draw] (v3) at (3,0) {1};
    \node [below=8pt] at (v2)  {$v_2$};
          \node [below=8pt] at (v1)  {$v_1$};
            \node [below=8pt] at (v2)  {$v_2$};
          \node [below=8pt] at (v3)  {$v_3$};
  \draw  (v1)to [bend left=30] node [below]{\tiny{${e_1}$}} (v2);
  \draw  (v1)to [bend right=30] node [below]{\tiny{${e_3}$}} (v2);
  \draw (v1) to node [below]{\tiny{$ {e_2}$}}  (v2); \draw (v3) to node [below]{\tiny{$ {e_4}$}}  (v2);
  \node at (1.5,0.9){(I)};
\end{tikzpicture}\\

\begin{tikzpicture}
\usetikzlibrary{positioning}
  \node [shape=circle,draw] (v1) at (0,0) {$\mbox{~~}$  };
  \node [shape=circle,draw] (v3) at (1,1) {1};
  \node [shape=circle,draw] (v2) at (2,0) {$\mbox{~~}$ };
    \node [below=8pt] at (v2)  {$v_2$};
          \node [below=8pt] at (v1)  {$v_1$};
                  \node [right=8pt] at (v3)  {$v_3$};
  \draw  (v1)to [bend left=15] node [below]{\tiny{${e_1}$}}(v2);
  \draw  (v1)to [bend right=15] node [below]{\tiny{${e_2}$}} (v2);
  \draw (v1) to node [above]{\tiny{${e_3}$}} (v3) to node [right]{\tiny{${e_4}$}} (v2);
  \node at (1,1.7){(J)};
\end{tikzpicture}
~~~
\begin{tikzpicture}
\usetikzlibrary{positioning}
  \node [shape=circle,draw] (v1) at (0,0) {$\mbox{~~}$ };
  \node [shape=circle,draw] (v2) at (0,1.6) {$\mbox{~~}$ };
  \node [shape=circle,draw] (v3) at (1,0.8) {$\mbox{~~}$ };
  \node [shape=circle,draw] (v4) at (2,0.8) {1};
   \node [left=8pt] at (v2)  {$v_2$};
          \node [left=8pt] at (v1)  {$v_1$};
                  \node [below=8pt] at (v3)  {$v_3$};
                    \node [below=8pt] at (v4)  {$v_4$};
  \draw  (v1)to [bend left=15]  node  [left] {\tiny{${e_1}$}} (v2);
  \draw  (v1)to [bend right=15]  node [right=-0.4pt]{\tiny{${e_2}$}} (v2);
  \draw (v1) to node  [below] {\tiny{${e_3}$}}  (v3) to node  [above] {\tiny{${e_4}$}}  (v2);\draw (v3) to node  [below] {\tiny{${e_5}$}}  (v4);
  \node at (1,2){(K)};
\end{tikzpicture}
~~~
\begin{tikzpicture}
\usetikzlibrary{positioning}
  \node [shape=circle,draw] (v1) at (0,0.5) { $\mbox{~~}$};
  \node [shape=circle,draw] (v2) at (2,0.5) { $\mbox{~~}$};
     \node [below=8pt] at (v2)  {$v_2$};
          \node [below=8pt] at (v1)  {$v_1$};
  \draw  (v1)to [bend left=15] node  [below] {\tiny{${e_2}$}} (v2);
  \draw  (v1)to [bend right=15] node  [below] {\tiny{${e_3}$}}  (v2);
 \draw  (v1)to [bend left=45] node  [below] {\tiny{${e_1}$}} (v2);
  \draw  (v1)to [bend right=45] node  [below] {\tiny{${e_4}$}} (v2);
  \node at (01,1.7){(L)};
\end{tikzpicture}
~~~
\begin{tikzpicture}
\usetikzlibrary{positioning}
  \node [shape=circle,draw] (v1) at (0,0) { $\mbox{~~}$};
  \node [shape=circle,draw] (v2) at (2,0) { $\mbox{~~}$};
  \node [shape=circle,draw] (v3) at (1,1) { $\mbox{~~}$};
    \node [below=8pt] at (v2)  {$v_2$};
          \node [below=8pt] at (v1)  {$v_1$};
                  \node [right=8pt] at (v3)  {$v_3$};
  \draw  (v1)to [bend left=15] node  [left] {\tiny{${e_1}$}}(v3);
  \draw  (v1)to [bend right=15] node  [right,below] {\tiny{${e_2}$}}(v3);
 \draw  (v2)to [bend left=15] node  [left] {\tiny{${e_3}$}}(v3);
  \draw  (v2)to [bend right=15] node  [right] {\tiny{${e_4}$}}(v3);\draw (v1)to  node  [below] {\tiny{${e_5}$}} (v2);
  \node at (1,1.7){(M)};
\end{tikzpicture}\\

\begin{tikzpicture}
\usetikzlibrary{positioning}
  \node [shape=circle,draw] (v3) at (0,0) { $\mbox{~~}$};
  \node [shape=circle,draw] (v4) at (1.5,0) { $\mbox{~~}$};
  \node [shape=circle,draw] (v1) at (0,1.5) { $\mbox{~~}$};
  \node [shape=circle,draw] (v2) at (1.5,1.5) { $\mbox{~~}$};
  \node [right=8pt] at (v2)  {$v_2$};
          \node [left=8pt] at (v1)  {$v_1$};
                  \node [below=8pt] at (v3)  {$v_3$};
                   \node [below=8pt] at (v4)  {$v_4$};
  \draw  (v1)to [bend left=15] node  [above] {\tiny{${e_1}$}}(v2);
  \draw  (v1)to [bend right=15] node  [below] {\tiny{${e_2}$}}(v2);
 \draw  (v4)to [bend left=15] node [below]  {\tiny{${e_4}$}}(v3);
  \draw  (v4)to [bend right=15] node [above]  {\tiny{${e_3}$}} (v3);\draw (v1)to node [left]  {\tiny{${e_5}$}} (v3);
  \draw (v2)to  node [right]  {\tiny{${e_6}$}} (v4);
  \node at (0.7,2.5){(N)};
\end{tikzpicture}
~~~
\begin{tikzpicture}
\usetikzlibrary{positioning}
  \node [shape=circle,draw] (v2) at (0,0) { $\mbox{~~}$};
  \node [shape=circle,draw] (v3) at (2,0) { $\mbox{~~}$};
  \node [shape=circle,draw] (v4) at (1,1) { $\mbox{~~}$};
  \node [shape=circle,draw] (v1) at (1,2.2) { $\mbox{~~}$};
   \node [below=8pt] at (v2)  {$v_2$};
          \node [right=8pt] at (v1)  {$v_1$};
                  \node [below=8pt] at (v3)  {$v_3$};
                   \node [below=8pt] at (v4)  {$v_4$};
  \draw  (v1)to node [left]  {\tiny{${e_4}$}} (v2)to node [below]  {\tiny{${e_5}$}} (v3)to node [left,below]  {\tiny{${e_3}$}}
  (v4) to node [below=5pt,left=-4pt]  {\tiny{${e_1}$}} (v1) to node [right]  {\tiny{${e_6}$}} (v3);\draw (v4)to (v2);
  \node at (1,3){(O)};
\end{tikzpicture}}
\end{lemma}
\begin{proof}
  See \cite[Lemma 3.2]{AI02}.
\end{proof}

\begin{theorem}\label{thmType1}
  Let $\phi\in\mathrm{Mod}({\Sigma_3})$ be a pseudo-periodic map of negative twist, and $\gamma$ be a cut curve of type 1, then
  $$|c(\phi,\gamma)|\geq\frac1{30},$$
  and the equality holds if and only if $(G(F_\phi),\gamma)$ is one of figures (3-1a), (3-1b) and (3-1c).
\end{theorem}
\begin{proof}
The possible cut curves are in (B), (C), (D), (F), (G), (I), and (K).

 \underline{ Case (B)} In this case, $\gamma$ is $e_1$ and $m_{\gamma}=1$. Similar to Claim 1 in the proof of Theorem \ref{thmType1-g=2}, we may assume that $g(B_{v_1})\leq2, g(B_{v_2})\leq1$, and thus $n(B_{v_1})\leq 10$ and $n(B_{v_2})\leq 6$ classified in Lemma \ref{lemmaValencies}. By (\ref{eqFDCTnon-a}),
  $$|c(\phi,\gamma)|=\frac{\mu_{v_1,\gamma}}{\lambda_{v_1,\gamma}}+\frac{\mu_{v_2,\gamma}}{\lambda_{v_2,\gamma}}+K \geq  \frac1{30}.$$
If the equality holds, then $K=-1$. Furthermore, by  Lemma \ref{lemmaValencies}, the cut curves and the pseudo-periodic maps with the lowest bound are classified  as follows:
\vspace{2mm}

(3-1a)~ $B_{v_1}: {\bf\frac3{10}}+\frac15+\frac12, ~~B_{v_2}: {\bf\frac13}+\frac13+\frac13$, $K=-1$; \vspace{1mm}

(3-1b)~ $B_{v_1}: {\bf\frac15}+\frac25+\frac25, ~~B_{v_2}: {\bf\frac56}+\frac23+\frac12$, $K=-1$; \vspace{1mm}

(3-1c)~ $B_{v_1}: {\bf\frac15}+\frac15+\frac35, ~~B_{v_2}: {\bf\frac56}+\frac23+\frac12$, $K=-1$.
\vspace{2mm}

\noindent The dual graphs of the above three pseudo-periodic maps are Figure (3-1a), Figure (3-1b), and Figure (3-1c) respectively.

 \underline{Case (C)} We may assume that $\gamma$ is $e_1$. Then $m_{\gamma}\leq2$. We may assume that $g(B_{v_i})=1$ and $n(B_{v_i})\leq 6$, where $i=1,3$. Then, by  Lemma \ref{lemmaValencies} (i) and (\ref{eqFDCTnon-a}), we have that
$$|c(\phi,\gamma)| \geq\frac1{2}\Big(\frac{\mu_{v_1,\gamma}}{\lambda_{v_1,\gamma}}+
\frac{\mu_{v_3,\gamma}}{\lambda_{v_3,\gamma}}+K\Big)\geq \frac1{24}.$$

\underline{Case (D)}  In this case, we may assume that $\gamma$ is $e_1$. If $m_\gamma\geq2$, then $n(B_{v_4})=m_\gamma\leq3$, and $\lambda_{v_4,\gamma}=n(B_{v_4})/m_\gamma=1$; if $m_\gamma=1$, then $\lambda_{v_4,\gamma}=n(B_{v_4})\leq2$. Since $\lambda_{v_1,\gamma}\leq 6$, we know that
$$|c(\phi,\gamma)| \geq\frac1{m_\gamma}\Big(\frac{\mu_{v_4,\gamma}}{\lambda_{v_4,\gamma}}+
\frac{\mu_{v_1,\gamma}}{\lambda_{v_1,\gamma}}+K\Big)
\geq \frac1{18}.$$


The rest cases are similar, and we omit their proofs.

\end{proof}
\begin{corollary}\label{coro3-1}
Let $f$ be a fibration of genus $3$ with $\delta_1(f)\neq0$, then
  $$\delta_1(f)\geq\frac1{30},$$
  and the equality holds if and only if  all the singular fibers of $f$ have smooth reduction except one whose dual graph is one of  figures (3-1a), (3-1b) and (3-1c).
\end{corollary}
\begin{proof}
Similar to the proof of Corollary \ref{coro2-1}.
\end{proof}

\begin{lemma}\label{lemmaType0}
  Let $\phi\in\mathrm{Mod}({\Sigma_3})$ be a pseudo-periodic map of negative twist, and $\gamma$ be a cut curve of type 0 which is adjacent to one component only. Then
  $$|c(\phi,\gamma)|\geq\frac1{5},$$
  and the equality holds if and only if $(G(F_\phi),\gamma)$ is one of the following figures (3-0-0a), (3-0-0b) and (3-0-0c).

{\upshape
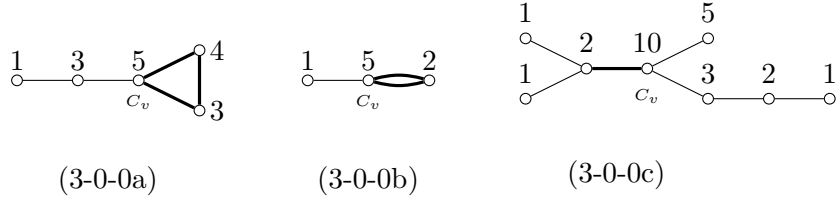
\begin{center}
\begin{tikzpicture}
  [inner sep=0.5mm,
  place/.style={circle,draw}]
  \node[place] (v2) at (0.8,0.8) [label=above:3]  {};
    \node[place] (v3) at (1.6,0.8) [label=above:5]  {};
  \node[place] (v1) at (0,0.8) [label=above:1]  {};
      \node[place] (v4) at (2.4,1.2) [label=right:4]  {};
  \node[place] (v5) at (2.4,0.4) [label=right:3]  {};
  \draw[very thick]  (v3) to  (v4) to (v5) to (v3);
  \draw (v1) to  (v2) to (v3);
\node [below=4pt] at (v3)  {\tiny $C_{v}$};
  \node at (1.2,0) [below=8pt] {(3-0-0a)};
\end{tikzpicture}
~~~~~
\begin{tikzpicture}
  [inner sep=0.5mm,
  place/.style={circle,draw}]
  \node[place] (v1) at (0.8,0.8) [label=above:5]  {};
    \node[place] (v2) at (1.6,0.8) [label=above:2]  {};
  \node[place] (v3) at (0,0.8) [label=above:1]  {};
  \draw[very thick]  (v1) to [bend left=15]   (v2);
  \draw[very thick] (v1) to [bend right=15]  (v2);
  \draw (v3) to (v1); \node [below=4pt] at (v1)  {\tiny $C_{v}$};
  \node at (0.8,0) [below=8pt] {(3-0-0b)};
\end{tikzpicture}
~~~~~
  \begin{tikzpicture}
  [  place/.style={circle,draw,inner sep=0.5mm},
  place2/.style={circle,draw,fill,inner sep=0.2mm},]
  \node[place] (v1) at (0.8,0.8) [label=above:2]  {};
    \node[place] (v2) at (1.6,0.8) [label=above:10]  {};
    \node[place] (v3) at (0,1.2) [label=above:1]  {};
  \node[place] (v4) at (0,0.4) [label=above:1]  {};
  \node[place] (v5) at (2.4,1.2) [label=above:5]  {};
   \node[place] (v6) at (2.4,0.4) [label=above:3]  {};
   \node[place] (v7) at (3.2,0.4) [label=above:2]  {};
   \node[place] (v8) at (4,0.4) [label=above:1]  {};
  \draw (v3)-- (v1)--(v2)--(v5);\draw (v4)--(v1); \draw  (v2)--(v6)--(v7)--(v8);
   \node at (1.2,0) [below=8pt] {(3-0-0c)};
    \draw[very thick]  (v1)  to (v2); \node [below=4pt] at (v2)  {\tiny $C_{v}$};
\end{tikzpicture}
\end{center}}
\end{lemma}
\begin{proof}
The proof is similar to that of  Case (I) in Theorem \ref{thmType0-g=2}.

Assume $\gamma$ is adjacent to $B_v$ with $g(B_v)+\rho(v)\leq3$ and $\rho(v)\geq1$.  Let $(m_1,\lambda_1,\sigma_1)$ and $(m_2,\lambda_2,\sigma_2)$ be the valencies of the boundary curves of $\SA_\gamma$.

   If $\gamma$ is non-amphidrome,  then $\lambda_1=\lambda_2=n(B_v)/m_\gamma$. Since $\rho(v)\geq1$ and $g=3$,  by Lemma \ref{lemmaValencies}, we have that
   $$|c(\phi,\gamma)|=\frac1{m_\gamma}(\frac{\mu_1}{\lambda_1}+\frac{\mu_2}{\lambda_2}+K) \geq\frac15.$$
   Furthermore, the equality holds if and only if  $K=-1$, $g(B_v)=2,m_\gamma=\rho(v)=1$, and the valencies of the boundary curves of $\SA_\gamma$ are one of the following cases:

\vspace{2mm}
(3-0-0a) $B_v: {\bf\frac45}+{\bf\frac35}+\frac 35$.
\vspace{1mm}

(3-0-0b) $B_v: {\bf\frac25}+{\bf\frac25}+\frac15$.
\vspace{1mm}

   If $\gamma$ is amphidrome, then the valencies of the boundary curves of $\SA_\gamma$ are the same $(2m_\gamma,\lambda,\sigma)$. Here $\lambda=n(B_v)/2m_\gamma$. Since $\rho(v)\geq1$, $g(B_v)\leq 2$, we have $n(B_v)\leq10$. Thus
    $$|c(\phi,\gamma)|=\frac1{m_\gamma}(\frac{\mu}{\lambda}+K) \geq\frac2{n(B_v)}\geq\frac1{5}.$$
 Furthermore, the equality holds if and only if $K=0$,  $g(B_v)=2,m_\gamma=\rho(v)=1$, and the valencies of the boundary curves of $\SA_\gamma$ are

\vspace{2mm}
(3-0-0c) $B_v: {\frac3{10}}+{\bf\frac 15}+{\frac12}$.
\vspace{1mm}

\end{proof}

\begin{theorem}\label{thmType0}
   Let $\phi\in\mathrm{Mod}({\Sigma_3})$ be a pseudo-periodic map of negative twist, and $\gamma$ be a cut curve of type 0. Then
  $$|c(\phi,\gamma)|\geq\frac1{12},$$
  and the equality holds if and only if $(G(F_\phi),\gamma)$ is one of figures (3-0a), (3-0b), (3-0c).
\end{theorem}
\begin{proof}

By Lemma \ref{lemmaType0}, we may assume that $\gamma$ is adjacent to two connected components. Thus $\gamma$ is in (E)-(O).

\underline{Case (E)} Without loss of generality, we may assume $\gamma=e_1$. Then $m_\gamma\leq2$.

\noindent({\bf{E1}}). Assume that $m_\gamma=1$. If $\phi$ interchanges $B_{v_i}$, then $\gamma$ is amphidrome. So $\lambda_{v_i,\gamma}= n(B_{v_i})$ for $i=1,2$  by (\ref{lami}), and $n(B_{v_i})\leq4$ by Lemma \ref{lemmaValencies} (i).  Thus
$|c(\phi,\gamma)|\geq\frac14$ by (\ref{eqFDCT-a}).
If $\phi$ does not interchanges $B_{v_i}$, then $\gamma$ is non-amphidrome. So $\lambda_{v_i,\gamma}= n(B_{v_i})$ for $i=1,2$, and $n(B_{v_i})\leq4$. Thus, by (\ref{eqFDCTnon-a}),
$$|c(\phi,\gamma)| =\frac{\mu_{v_1,\gamma}}{\lambda_{v_1,\gamma}}+\frac{\mu_{v_2,\gamma}}{\lambda_{v_2,\gamma}}+K \geq\frac1{12}.$$
Furthermore, the equality holds if and only if $K=-1$ and the valencies of the boundary curves of $\SA_\gamma$ are

\vspace{2mm}
(3-0a) $B_{v_1}: {\bf\frac34}+\frac 34+\frac12,~~ B_{v_2}: {\bf\frac13}+\frac13+\frac13$.
\vspace{1mm}

\noindent({\bf{E2}}). Assume that $m_\gamma=2$.  If $\phi$ interchanges $B_{v_i}$, then  $\gamma$ is non-amphidrome, and $\lambda_{v_i,\gamma}=2n(B_{v_i})/m_\gamma=n(B_{v_i})~(i=1,2)$  by (\ref{lami}). Similarly, $\lambda_{v_i,e_2}=n(B_{v_i})$, and thus $\lambda_{v_i,e_j}=n(B_{v_i})$ for $i=1,2, j=1,2$. By Lemma \ref{lemmaValencies} (i), we have that $n(B_{v_i})\leq4$. So $|c(\phi,\gamma)|\geq\frac18$  by (\ref{eqFDCTnon-a}). If $\phi$ does not interchange $B_{v_j}$, then  $\gamma$ is also non-amphidrome, and $\lambda_{v_i,\gamma}=n(B_{v_i})/2\leq3$. Thus
$$|c(\phi,\gamma)| =\frac12\Big(\frac{\mu_{v_1,\gamma}}{\lambda_{v_1,\gamma}} +\frac{\mu_{v_2,\gamma}}{\lambda_{v_2,\gamma}}+K\Big) \geq\frac1{12}.$$

\noindent Furthermore, the equality holds if and only if  $K=-1$ and the valencies of the boundary curves of $\SA_\gamma$ are one of the following two cases:

\vspace{2mm}

(3-0b) $B_{v_1}: \frac56+{\bf\frac 23}+\frac12, ~~B_{v_2}: \frac14+\frac14+{\bf\frac12}$;\vspace{1mm}

(3-0c)  $B_{v_1}: \frac56+{\bf\frac 23}+\frac12, ~~B_{v_2}: \frac34+\frac34+{\bf\frac12}$.
\vspace{2mm}

 Though the rest cases are similar as above, we give the detail proof here for the reader's convenience.

\underline{Case (F)} We may assume that $\gamma=e_1$.  In this case, the rational component $B_{v_i}~(i=3,4)$ is adjacent to three edges,  $m_{\gamma}\leq 2$,  and $n(B_{v_i})\leq2$. If $m_\gamma=1$, then $\lambda_{v_i,\gamma}=n(B_{v_i})=1$ by (\ref{lami}). If $m_\gamma=2$, then  $\lambda_{v_i,\gamma}\leq 2n(B_{v_i})/m_\gamma\leq2$ by (\ref{lami}).   So, we always have that $\lambda_{v_i,\gamma}\leq2$. By (\ref{eqFDCTnon-a}) and (\ref{eqFDCT-a}),
$$|c(\phi,\gamma)|\geq\frac14.$$

\underline{Case (G)} We may assume that $\gamma=e_1$. Then $m_\gamma\leq2$.

\noindent({\bf{G1}}). If $m_\gamma=1$, then $n(B_{v_2})=1$ and $\lambda_{v_2,\gamma}=1$. Since $\lambda_{v_1,\gamma}=\lambda_{v_1,e_2}=n(B_{v_1})$,    we know that $\lambda_{v_1,\gamma}=n(B_{v_1})\leq4$. Thus
$$|c(\phi,\gamma)|\geq\frac14.$$

\noindent({\bf{G2}}). If $m_{\gamma}=2$, then $n(B_{v_2})=2$ and $\lambda_{v_2,\gamma}=1$. Since $\lambda_{v_1,\gamma}=\lambda_{v_1,e_2}=\frac12 n(B_{v_1})\leq3$, we know that
$$|c(\phi,\gamma)|\geq\frac16.$$

\underline{Case (H)} We may assume that $\gamma=e_1$. Then $m_\gamma\leq3$.

\noindent({\bf{H1}}). If $m_{\gamma}\geq2$,  then $n(B_{v_2})=m_\gamma$, $\lambda_{v_2,\gamma}=n(B_{v_2})/m_\gamma=1$, and $\lambda_{v_1,\gamma}=n(B_{v_1})/m_\gamma$.  Since $n(B_{v_1})\leq6$, we have that
  $$|c(\phi,\gamma)|=\frac1{m_\gamma}(\frac{\mu_1}{\lambda_{v_1,\gamma}}+\frac{\mu_2}{\lambda_{v_2,\gamma}}+K)\geq\frac1{m_\gamma\lambda_{v_1,\gamma}}\geq\frac16.$$

\noindent({\bf{H2}}). If $m_{\gamma}=1$, then $\lambda_{v_2,\gamma}= n(B_{v_2})\leq2$ and  $\lambda_{v_1,\gamma}=n(B_{v_1})\leq6$. So
$$|c(\phi,\gamma)|\geq\frac16.$$

\underline{Case (I)} We may assume that $\gamma=e_1$. Then $m_\gamma\leq3$.
 In this case, we always have that $\lambda_{v_1,\gamma}=\lambda_{v_2,\gamma}\leq2$, similarly as Case (H). So
 $$|c(\phi,\gamma)|\geq\frac16.$$

\underline{ Case (J)} If $\gamma$ is either $e_1$ or $e_2$, then $|c(\phi,\gamma)|\geq\frac14$, similarly as Case (F). Now we may assume that $\gamma$ is $e_3$. Then $m_\gamma\leq2$.

\noindent({\bf{J1}}).  If $\phi$ does not interchange $B_{v_i}~(i=1,2)$, then $m_\gamma=1$, $\lambda_{v_1,\gamma}=n(B_{v_1})\leq2$ and $\lambda_{v_3,\gamma}\leq6$. So $$|c(\phi,\gamma)|\geq\frac16.$$

 \noindent({\bf{J2}}). Assume that $\phi$ interchanges $B_{v_i}$. Then $m_\gamma=2$ and $\lambda_{v_3,\gamma}=n(B_{v_3})/2\leq3$. In this case, we always have that $\phi^2(\vec e_i)=\vec e_i~(i=1,2)$. Thus $n(B_{v_i})=1~(i=1,2)$ and $\lambda_{v_1,\gamma}=1$.  So
 $$|c(\phi,\gamma)|\geq\frac16.$$

 Similarly, we can obtain results for Cases (K), (M), (N) and (O).

\underline{ Case (L) } We may assume that $\gamma=e_1$. Then $m_\gamma\leq4$.

\noindent({\bf{L1}}). Assume that $m_{\gamma}\geq2$.
If $m_\gamma$ is even and $\phi$ interchanges $B_{v_i}$, then $\gamma$ is non-amphidrome, $n(B_{v_i})=m_\gamma/2$, and $\lambda_{v_i,\gamma}= 2n(B_{v_i})/m_\gamma=1$ by (\ref{lami}).  If  $m_\gamma$ is even and $\phi$ does not interchanges $B_{v_i}$, then  $\gamma$ is also non-amphidrome, $n(B_{v_i})=m_\gamma$, and $\lambda_{v_i,\gamma}= n(B_{v_i})/m_\gamma=1$.
 If $m_\gamma$ is odd, that is, $m_\gamma=3$, then $n(B_{v_i})=3$ and $\lambda_{v_i,\gamma}=n(B_{v_i})/m_\gamma=1$  by (\ref{lami}).  So, we always have that $\lambda_{v_i,\gamma}=1$ and  $|c(\phi,\gamma)|\geq\frac14.$

\noindent({\bf{L2}}). If $m_{\gamma}=1$, then $\lambda_{v_1,\gamma}\leq n(B_{v_i})\leq3$.  So  $|c(\phi,\gamma)|\geq\frac13.$
 \end{proof}

\begin{corollary}\label{coro3-0}
 Let $f$ be a fibration of genus $3$ with $\delta_0(f)\neq0$, then
  $$\delta_0(f)\geq\frac1{6},$$
  and the equality holds if and only if  all the singular fibers of $f$ have smooth reduction except one whose dual graph is one of  figures (3-0a), (3-0b) and (3-0c).
\end{corollary}
\begin{proof}
Let $F$ be a singular fiber of $f$, then we claim that:

Claim: $\delta_0(F)\geq\frac16$, and $\delta_0(F)=\frac1{6}$ if and only if the dual graph of $F$ is one of figures (3-0a),(3-0b) and (3-0c).

Proof of Claim: By Lemma \ref{lemmaType0} and (\ref{eqdeltai}), we may assume that the dual graph of the stable model $\tilde F$ of $F$ is one of (E)--(O).
 For each of these graphs,  there are at least two cut curves of type 0 adjacent to two connected components. So, by (\ref{eqdeltai}) and Theorem \ref{thmType0}, we have
  $$\delta_0(F)=\sum_{\gamma\in\SC_{\phi_F,0}}|c(\phi,\gamma)|\geq2\cdot\frac1{12}=\frac16.$$
  If $\delta_0(F)=\frac1{6}$, then $F$ has exactly two cut curves $\gamma$ of type 0 with $|c(\phi,\gamma)|=\frac1{12}$. Thus we obtain Claim.

  The rest of the proof  is similar to that of Corollary \ref{coro2-1}.
\end{proof}

\begin{remark}
   From the proof of Theorem \ref{thmType0}, we know that if $\gamma$ is a cut curve of type 0 in (F)--(O), and $\gamma$ is adjacent to two connected components, then $|c(\phi,\gamma)|\geq\frac16$. Similarly to the proof of Corollary \ref{coro3-0}, we have that:

  If the dual graph of the stable model $\tilde F$ of $F$ is one of (F)--(O), then $\delta_0(F)\geq\frac13$.
\end{remark}

Now we can give an immediate application of the above results to the effective Bogomolov conjecture \cite{LT17}.

\begin{proof}[Proof of Theorem \ref{thmBogo}]
Denote by $f$ the family of curves corresponding to $C/K$. By the assumption, we know that the semistable model of $f$ is not smooth. So either $\delta_0(f)\neq0$ or $\delta_1(f)\neq0$.  By \cite[Theorem 2.4]{Ci11},
\begin{equation*}
\inf_{D\in\mbox{Div}^1(\bar{C})}a'(D)\geq
\frac{1}{2(2g+1)}\Big(\frac{(g-1)^2}{2g(7g+5)}{\delta_0( f)}+\sum\limits_{i\in(0,g/2]}\frac{2i(g-i)}{g}{\delta_i(f)}\Big).
\end{equation*}
If $\delta_1(f)\neq0$, then
  \begin{equation*}
\inf_{D\in\mbox{Div}^1(\bar{C})}a'(D)\geq
    \begin{cases}
      \frac1{120}, &\mbox{ if } g=2,\\
      \frac1{315}, &\mbox{ if } g=3.
    \end{cases}
  \end{equation*}
 If $\delta_0(f)\neq0$, then
  \begin{equation*}
\inf_{D\in\mbox{Div}^1(\bar{C})}a'(D)\geq
    \begin{cases}
      \frac1{2280}, &\mbox{ if } g=2,\\
      \frac1{3276}, &\mbox{ if } g=3.
    \end{cases}
  \end{equation*}

  Comparing these two cases, we get our result.

\end{proof}

\section{Lower bounds of modular invariants for genus 2}\label{sectMaing=2}

\subsection{Lower bounds}

Now we can use results in Section \ref{sectFDCTbounds} to prove Theorem \ref{thmlambdakappa}, Theorem \ref{thmg3mod}, and Theorem \ref{thmg4mod}.

\begin{proof}[Proof of Theorem \ref{thmlambdakappa}]

Since $f$ is a family of curves of genus 2, we know that (\cite[p.317]{Mu83})
 \begin{equation}\label{formula}
\lambda(f)=\frac1{10}\delta_0(f)+\frac15\delta_1(f),~~~~
\kappa(f)=\frac15\delta_0(f)+\frac75\delta_1(f).
\end{equation}

Because $f$ is non-isotrivial,  either $\delta_0(f)\neq0$ or $\delta_1(f)\neq0$. So there are the following two cases.

Case 1: $\delta_0(f)\neq0$. By  (\ref{formula}) and Corollary \ref{coro2-0}, we know that
\begin{equation}\label{equproof3.1}
\lambda(f)\geq\frac1{10}\delta_0(f)\geq\frac1{10}\times\frac13=\frac1{30},~~
\delta(f)\geq\delta_0(f)\geq\frac13,~~
\kappa(f)\geq\frac1{5}\delta_0(f)\geq\frac15\times\frac13=\frac1{15}.
\end{equation}

Case 2: $\delta_1(f)\neq0$. Similarly, we have that
$$\lambda(f)\geq\frac15\times\frac1{12}=\frac1{60},~~\delta(f)\geq\frac1{12},~~\kappa(f)\geq\frac75\times\frac1{12}=\frac7{60}.$$
 From the above two cases, it is easy to see that
\begin{align*}
\lambda(f)\geq\mathrm{min}\{\frac1{30},\frac1{60}\}=\frac1{60},~~
\delta(f)\geq\mathrm{min}\{\frac1{3},\frac1{12}\}=\frac1{12},~~
\kappa(f)\geq\mathrm{min}\{\frac1{15},\frac7{60}\}=\frac1{15}.
\end{align*}
Moreover, $\lambda(f)=\frac{1}{60}$ if and only if $\delta(f)=\frac{1}{12}$ if and only if $\delta_1(f)=\frac1{12}$ and $\delta_0(f)=0$. So we obtain  Theorem \ref{thmlambdakappa} (1) by Corollary \ref{coro2-1}. Similarly, we can get Theorem \ref{thmlambdakappa} (2).

Now we have completed the proof except that each equality of (\ref{equationlambda}) can be reached, which will be proved by examples in Section \ref{sectionexistence}.

\end{proof}

\begin{remark}\label{remarklambdakappa}
  From Theorem \ref{thmlambdakappa}, we know that if $f$ has $\lambda(f)=\frac1{60}$, then $\delta(f)=\frac1{12}$, and $\kappa(f)=\frac7{60}\neq \frac1{15}$ by Noether equality. Hence there does not exist a non-isotrivial family $f$ such that both $\lambda(f)$ and $\kappa(f)$ are minimal.
\end{remark}

\subsection{Proof of rigidity, Theorem \ref{thmrigidity}}

Now we want to study rigidity properties of non-isotrivial families of curves with minimal modular invariants.

Let $f:S\to C$ be a relative minimal fibration of genus $g\geq2$, and  $C$ be a smooth curve of genus $b$. We have three fundamental relative
invariants which are non-negative,
\begin{equation}
\begin{split}
K_f^2&=K_{S/C}^2=K_S^2-8(g-1)(b-1), \\
e_f&=\chi_{\mathrm{top}}(S)-4(g-1)(b-1),\\
\chi_f&=\deg f_*\omega_{S/C}=\chi(\mathcal O_S)-(g-1)(b-1).
\end{split}
\end{equation}
If $f$ is semistable, then
\begin{equation}
\lambda(f)=\chi_f,~~\delta(f)=e_f,~~\kappa(f)=K_f^2.
\end{equation}
Moreover, if $f$ is semistable and $e_f\neq0$, then $\chi_f$ and $K_f^2$ are positive.

Chern numbers $c_1^2(F),c_2(F),\chi_F$ of a singular fiber $F$ are defined as follows (see \cite{Ta10}),
 \begin{equation}\label{c1c2}
\begin{cases}
c_1^2(F)= 4N_F+F_{\mathrm{red}}^2+\alpha_F-\beta_F^-,&\\
c_2(F)=2N_F+\mu_F-\beta_F^+,&\\
12\chi_F=6N_F+F_{\mathrm{red}}^2+\alpha_F+\mu_F-\beta_F.
\end{cases}
\end{equation}
For the notations here, we refer to \cite[\S 1]{Ta10} and \cite[\S 2]{LT13a}.
By (\ref{c1c2}), it is easy to get Chern numbers of the extremal fibers  in Theorem \ref{thmlambdakappa}. For similar detailed computation, we refer to \cite{LT13a}. In the following, denote by $F_{2-1a}$  the singular fiber whose dual graph is Figure (2-1a) in Theorem \ref{thmlambdakappa}, and others are similar.
 {\renewcommand{\arraystretch}{1.5}
 \begin{figure}[h]
\begin{center}
\begin{tabular}{|c|c|c|c|c|c|c|c|c|c|c|}
\hline
 $F$  & $N_F$ & $F_{\mathrm{red}}^2$ & $\mu_F$ & $\alpha_F$ & $\beta_F^-$ & $\beta_F^+$   & $\beta_F$   & $\chi_F$ & $c_1^2(F)$ & $c_2(F)$\\
\hline
$F_{2-1a}$  & $2$ & $-4$ & $7$ & $0$ & $\frac{23}{12}$ & $\frac1{12}$  & $2$   & $\frac{13}{12}$ & $\frac{25}{12}$ & $\frac{131}{12}$
\\\hline
$F_{2-1b}$  & $2$ & $-4$ & $7$ & $0$  & $\frac{23}{12}$ & $\frac1{12}$  & $2$ & $\frac{13}{12}$ & $\frac{25}{12}$ & $\frac{131}{12}$
\\\hline
$F_{2-0a}$ & $1$ & $-2$ &  5 & 1 & $\frac5{3}$ & $\frac1{3}$  & $2$  & $\frac23$ & $\frac43$ & $\frac{20}{3}$ \\\hline
$F_{2-0b}$ & $1$ & $-1$ &  4 & 0  & $\frac2{3}$ & $\frac1{3}$  & $1$  & $\frac23$ & $\frac73$ & $\frac{17}{3}$
\\\hline
\end{tabular}
\end{center}
\caption{Chern numbers of extremal singular fibers \label{chenextremal}}
 \end{figure}}

The relative invariants can be obtained from modular invariants and Chern numbers of singular fibers (see \cite{Ta94, Ta96}), i.e.,
\begin{equation}\label{modinv}
\begin{cases}
K_f^2=\kappa(f)+\sum_{i=1}^sc_1^2(F_i),&\\
e_f=\delta(f)+\sum_{i=1}^sc_2(F_i),&\\
\chi_f=\lambda(f)+\sum_{i=1}^s\chi_{F_i}.
\end{cases}
\end{equation}

Let $F$ be a singular fiber of genus 2 with smooth reduction, then $F$ is of elliptic type [1] in \cite{NU73}. The Chern numbers of these fibers have been computed in \cite{GLT16}.
We rewrite the obtained table  in \cite[\S 5.1]{GLT16}  as Figure \ref{cheniso} in the following for conveniece, where $\chi_F=\frac1{12}(c_1^2(F)+c_2(F))$.

{\renewcommand{\arraystretch}{1.2}
\begin{figure}[h]
\begin{center}
\begin{tabular}{|c|c|c|c|c|c|c|}
\hline $F$ & $[\mathrm{\Rnm1^*_{0-0-0}}]$ & $[\mathrm{\Rnm2}]$ & $[\mathrm{\Rnm3}]$ & $[\mathrm{\Rnm4}]$ &
$[\mathrm{\Rnm5}]$ & $[\mathrm{\Rnm5^*}]$ \\
\hline
{${(c_1^2,c_2,\chi)}$} & $(2,10,1)$ & $(2,4,\frac12)$ & $(2,10,1)$ & $(3,9,1)$ & $(1,5,\frac12)$
& $(3,15,\frac32)$
\\\hline
$F$  & $[\mathrm{\Rnm6}]$  & $[\mathrm{\Rnm7}]$  & $[\mathrm{\Rnm7^*}]$ & [$\mathrm{\Rnm8}$-1] & [$\mathrm{\Rnm8}$-2] & $[\mathrm{\Rnm8}$-3] \\\hline
{${(c_1^2,c_2,\chi)}$}  & $(2,10,1)$ & $(1,5,\frac12)$ & $(3,15,\frac32)$ & $(\frac45,4,\frac25)$ & $(\frac{12}5,12,\frac65)$ & $(\frac{13}5,7,\frac45$)\\\hline
$F$  & [$\mathrm{\Rnm8}$-4] & [$\mathrm{\Rnm9}$-1]
& [$\mathrm{\Rnm9}$-2]  & [$\mathrm{\Rnm9}$-3] & [$\mathrm{\Rnm9}$-4] &
\\\hline
{${(c_1^2,c_2,\chi)}$}  & $(\frac{16}5,16,\frac85)$ & $(\frac85,8,\frac45)$
& $(\frac65,6,\frac35)$  & $(\frac{14}5,14,\frac75)$ & $(\frac{12}5,12,\frac65)$ &
\\\hline
\end{tabular}
\end{center}
\caption{Chern numbers of singular fibers of genus 2 with smooth reduction\label{cheniso}}
\end{figure}}

From Figure \ref{cheniso}, we know that $\chi_F\geq\frac25$. Furthermore if $\chi_F\neq\frac25,$ then $\chi_F\geq\frac12$.

\begin{proof}[Proof of Theorem \ref{thmrigidity}]

Since $S$ is a rational surface, $q(S)=p_g(S)=0, \chi(\CO_S)=1$, and $C\cong\mP^1$. Thus
$$\chi_f=\chi(\CO_S)+1=2.$$

 Since $\lambda(f)=\frac1{60}$, there is only one singular fiber, say $F_1$, whose dual graph is either Figure (2-1a) or Figure (2-1b) by Theorem \ref{thmlambdakappa}. From Figure \ref{chenextremal}, we know that $\chi_{F_1}=\frac{13}{12}$. By (\ref{modinv}),
$$2=\chi_f=\lambda(f)+\sum_{i=1}^s\chi_{F_i}=\frac1{60}+\frac{13}{12}+\sum_{i=2}^{s}\chi_{F_i}\geq\frac{11}{10}+(s-1)\frac25.$$
Hence $s\leq3$. If $f$ has only two singular fibers, then $f$ is isotrivial (see \cite{Be81}), and $\lambda(f)=0$. Therefore $f$ has three singular fibers exactly. We may assume that the singular fibers are over three fixed points of $\mP^1$ by projective transformation. Then the desired finiteness is from the solved Shafarevich conjecture (see \cite{Ar71,Pa68}).
\end{proof}

\section{Lower bounds of modular invariants for $g\geq3$}\label{sectMaing3}
In order to see our method for $g\geq3$ clearly, we prove Theorem \ref{thmg4mod} first.

\begin{proof}[Proof of Theorem \ref{thmg4mod}]
Since $\delta(f)\neq0$, there exists $0\leq i\leq [g/2]$ such that $\delta_i(f)\neq0$. From \cite[Theorem 1.4]{LT17}, we know that if $\delta_i(f)\neq0$, then
  \begin{equation*}
  \delta_i(f)\geq
    \begin{cases}
      \frac1{4g^2}, &\mbox{ if } i=0,\\
      \frac1{(4i+2)(4(g-i)+2)}, &\mbox{ if }i\geq1.
    \end{cases}
  \end{equation*}

In this proof, we use the following Moriwaki's inequality (see \cite[Theorem D]{Mo98})
\begin{equation}\label{moriwaki}
(8g+4)\lambda(f)\geq g\delta_0(f)+\sum_{i=1}^{[g/2]}4i(g-i)\delta_i(f).
\end{equation}

If $\delta_0(f)\neq0$, then, by (\ref{moriwaki}),
$$\lambda(f)\geq \frac1{8g+4}\cdot g\cdot \frac1{4g^2}.$$
If $\delta_i(f)\neq0$ for some $i\geq1$, then
$$\lambda(f)\geq \frac1{8g+4}\cdot 4i(g-i)\cdot\frac1{(4i+2)(4(g-i)+2)}.$$
Hence, combining all these cases, we have that
\begin{align*}
\lambda(f)&\geq\frac1{8g+4}\mathrm{min}\{g\cdot\frac1{4g^2},\frac{4(g-1)}{(4+2)(4(g-1)+2)},\ldots,\frac{4[\frac{g}2](g-[\frac{g}2])}{(4[\frac{g}2]+2)(4(g-[\frac{g}2])+2)}\}\\
&\geq\frac{1}{16g(2g+1)}.
\end{align*}
By Cornalba-Harris-Xiao's slope inequality (\cite[Theorem 2]{Xi87}), we have that
$$\kappa(f)\geq\frac{4g-4}g\lambda(f)\geq\frac{g-1}{4g^2(2g+1)}.$$

\end{proof}

Applied the proof of Theorem \ref{thmg4mod} to the case $g=3$, we get Theorem \ref{thmg3mod} by Corollary \ref{coro3-1} and Corollary \ref{coro3-0} directly. So we omit the proof of Theorem \ref{thmg3mod}.

\begin{proof}[Proof of Theorem \ref{thmg=3hyp}]
If $f:S\to C$ is a hyperelliptic fibration of genus $g\geq2$, then (\cite[(4.11)]{CH88})
\begin{equation}\label{CHeq}
(8g+4)\lambda(f)=g\xi_0(f)+\sum_{j=1}^{[{(g-1)}/2]}2(j+1)(g-j)\xi_j(f)+\sum_{i=1}^{[g/2]}4i(g-i)\delta_i(f).
\end{equation}
See \cite{CH88,Li16} for the notation $\xi_j$.

The hyperelliptic singular fibers $F_{(\mathrm{i}k)}$ of genus three with periodic monodromy are classified in \cite[Lemma 1.1]{Is04},  and we list them in Figure \ref{cheniso} using the same notations.
It is easy to know (\cite{Li16}) that $$\xi_j(F_a)=\xi_j(F_b)=\xi_j(F_{(\mathrm{i}k)})=0 ~(j\geq0), ~~\delta_1(F_a)=\delta_1(F_b)=\frac1{30},~~\delta_1(F_{(\mathrm{i}k)})=0.$$

Let $F_1,F_2,\ldots,F_s$ be all singular fibers of $f$. By our assumption, we may assume that $F_1$ is either $F_a$ or $F_b$. If $\lambda(f)=\frac1{105}$, then $\xi_0(F_l)=\xi_1(F_l)=\delta_1(F_l)=0 ~(l\geq2)$ by (\ref{CHeq}), i.e., $F_l~(l\geq2)$ have periodic monodromy. Moreover,
if $F_l~(l\geq2)$ have periodic monodromy, then  $\lambda(f)=\frac1{105}$  by (\ref{CHeq}). In this case, we have that
$$\xi_0(f)=0, ~~\delta_1(f)=\frac1{30},~~\xi_1(f)=0,$$
and
$$\chi_f=\lambda(f)+\sum_{l=1}^s \chi_{F_l}=\frac8{28}\cdot\frac1{30}+\sum_{l=1}^s \chi_{F_l}=\frac1{105}+\sum_{l=1}^s \chi_{F_l}.$$
  Since $S$ is rational, $\chi(\CO(S))=1$ and $C\cong\mP^1$. Then $\chi_f=\chi(\CO(S))-(g-1)(b-1)=3$.

 It is easy to calculate that
$\chi_{F_a}=\frac{17}{15},~~\chi_{F_b}=\frac{49}{30}$, and we give the Chern number $\chi_{F_{(\mathrm{i}k)}}$ for each fiber $F_{(\mathrm{i}k)}$   in  Figure \ref{cheniso}.
Thus we can classify all possible configurations directly.

{\renewcommand{\arraystretch}{1.2}
\begin{figure}[h]
\begin{center}
\begin{tabular}{|c|c|c|c|c|c|c|c|c|c|c|}
\hline $F$ & (i1) & (i2) & (i3) & (i4) &
(i5) & (i6)  & (i7)  & (i8)  & (i21) \\
\hline
{${\chi_F}$} & $\frac{25}{14}$ & $\frac{17}{14}$ & $\frac{33}{14}$ & $\frac{9}{14}$ & $\frac{29}{14}$
& $\frac{13}{14}$ & $\frac94$ & $\frac34$ & $\frac74$
\\\hline
$F$  & (i22) & (i23) & (i24) & (i25) & (i26)
& (i29)  & (i30) & (i31) & (i32)
\\\hline
{${\chi_F}$}  & $\frac54$  & $\frac94$ & $\frac34$ & $\frac{15}7$ & $\frac67$
& $\frac{12}7$  & $\frac97$ & $\frac{11}7$ & $\frac{10}7$
\\\hline
$F$  & (i33) & (i38) & (i39)  & (i40) & (i43) & (i44) & (i45) & (i46)
& (i47)
\\\hline
{${\chi_F}$}  & $\frac32$ & $\frac32$  & $\frac74$  & $\frac54$ & $\frac32$ & $1$ & $1$ & $1$
& $\frac12$
\\\hline
\end{tabular}
\end{center}
\caption{$\chi_F$ of hyperelliptic singular fibers of genus 3 with smooth reduction\label{cheniso}}
\end{figure}}

\end{proof}

\begin{remark}\label{remg4}

Note that the equality in (\ref{moriwaki}) holds for hyperelliptic fibrations $f$ with $\xi_j(f)=0~(j>0)$ (see (\ref{CHeq})). For $g\geq4$, if there exists a hyperelliptic singular fiber with minimal fractional Dehn twist coefficient of corresponding type,  then we shall obtain the sharp lower bound for $\lambda(f)$ which can be reached ``combinatorially",  similarly as $g=3$.
\end{remark}

\section{Proof of optimum, Theorem \ref{thmexistence}}\label{sectionexistence}
Before the proof of Theorem \ref{thmexistence}, we introduce the notation of ramification index.

  A reduced divisor $D$ of $S$ is called vertical, if $f(D)$ is a point. If $D$ contains no vertical component, then $f$ induces a morphism
$\phi:D\to C$. Let
$$\rho=\rho_1\circ\cdots\circ\rho_r:
(\tilde S,\tilde D)=(S_r,D_r)\stackrel{\rho_r}{\to}(S_{r-1},D_{r-1}){\to}\cdots
\stackrel{\rho_2}{\to}(S_1,D_1)\stackrel{\rho_1}{\to} (S_0,D_0)=(S,D)$$
be the resolution of $D$, where $D_i$ is the strict transform of $D_{i-1}$, $\tilde D$ is smooth, and
$\rho_i$ is a blow-up at a singularity of $D_{i-1}$ with multiplicity $m_i$. Then the {\it relative ramification index} of $\phi$ is defined to be
\begin{equation}\label{ramindex}
r(D):=\deg \tilde R+\sum_{i=1}^rm_i(m_i-1),
\end{equation}
where $\deg\tilde R$ is the ramification index of the induced morphism
 $\tilde\phi:\tilde D\to C$. Then
 \begin{equation}\label{RH}
   r(D)=K_{S/C}D+D^2,
 \end{equation}
which is a generalized Riemann-Hurwitz formula (see \cite[Lemma 2.4.8]{Xi92}).

Now we give our examples.
\begin{proposition}\label{proplambda}
There is a family of fibrations $(f_{\lambda,n}:S_n\to \mP^1)_{n\in\mN}$ of genus 2 with $\lambda(f_{\lambda,n})=\frac1{60}$, $\delta(f_{\lambda,n})=\frac1{12}$, satisfying that

1) $f_{\lambda,n}$  has $2n+3$ singular fibers;

2) the  image of $f_{\lambda,n}$  in $\overline{\CM}_g$ by the moduli  map $J:\mP^1\to \overline{\CM}_g$ is the same as that of $f_{\lambda,0}$, for each $n\in\mN$.
\end{proposition}

\begin{proof}
Let $\Gamma_t$ be the fiber over $t\in\mP^1$ of the second projection
$$p_2:P=\mathbb{P}^1\times\mathbb{P}^1\to \mP^1, ~~p_2((x,t))=t.$$
Let $R_h$ be the divisor on $P$,
whose affine equation is
\begin{equation}\label{hxt}
 h(x,t)=x^6+(15x^4+40x^3)t-(45x^2+24x)t^2+5t^3.
\end{equation}
 Let  $R_{\lambda,n}=R_h+\Gamma_\infty+\sum_{i=1}^{2n}\Gamma_i$, where $n\geq0$ and $\Gamma_i$'s are generic fibers of $p_2$.  Here, when $n=0$,  the sum means that there is no generic fiber. Then there is  an invertible sheaf  $\delta_{\lambda,n}$ with $\CO_{P}(R_{\lambda,n})\cong\delta_{\lambda,n}^{\otimes2}$, and there is a double cover $\pi_n:S'_n\to P$ whose branch locus is $R_{\lambda,n}$,  see \cite[\S I17]{BPV84}. Taking birational transforms, we can obtain a relative minimal fibration $f_{\lambda,n}:S_n\to \mP^1$ induced by the second projection $p_2$.\\

Case 1: $n=0$. Denote $f_{\lambda,0}$ (resp. $R_{\lambda,0}$) by $f_\lambda$ (resp. $R_\lambda$) for brief.\\

{\it Claim A: There are exactly three singular fibers $F_0=f^{-1}_\lambda(0)$, $F_{-1}=f^{-1}_\lambda(-1)$ and $F_{\infty}=f^{-1}_\lambda(\infty)$ in $f_\lambda$. Moreover, $F_0$ is the fiber in Theorem \ref{thmlambdakappa} (1), $F_{-1}$ is of type [\Rnm8-1] and $F_\infty$ is of type [\Rnm2], see \cite{NU73} for the notations.}\\

Assume Claim A firstly, then we have that
$$\lambda(f_{\lambda})=\frac1{60},~~\delta(f_\lambda)=\frac1{12},$$
by Theorem \ref{thmlambdakappa}, for both $F_{-1}$ and $F_\infty$  have smooth reduction.
Furthermore, $\kappa(f_\lambda)=\frac7{60}$ by Noether equality $12\lambda(f_\lambda)=\kappa(f_\lambda)+\delta(f_\lambda)$.
From Equation (\ref{modinv}) and the Chern numbers of fibers in Figure \ref{chenextremal} and Figure \ref{cheniso}, we know that
\begin{align*}
K^2_{f_{\lambda}}&=\kappa(f_\lambda)+c_1^2(F_0)+c_1^2(F_{-1})+c_1^2(F_\infty)=\frac7{60}+\frac{25}{12}+\frac45+2=5,\\
\chi_{f_{\lambda}}&=\lambda(f_\lambda)+\chi_{F_0}+\chi_{F_{-1}}+\chi_{F_\infty}=\frac1{60}+\frac{13}{11}+\frac25+\frac12=2.
\end{align*}
Since $K^2_{f_\lambda}=5<6$, $S_0$ is a ruled surface by Theorem 0.2 in \cite{TTZ05}. Furthermore, $S_0$ is a rational surface for $q(S_0)=q_{f_\lambda}=0$ by Theorem \ref{thmlambdakappa}.\\

Case 2: $n\geq1$.

Comparing $R_{\lambda,n}$ with $R_\lambda$, it is easy to see that $f_{\lambda,n}$ has $2n+3$ singular fibers, three of them are the same as singular fibers of $f_\lambda$ and the rest are all of type $[\mathrm{\Rnm1^*_{0-0-0}}]$. Hence
$$\lambda(f_{\lambda,n})=\frac1{60},~~\kappa(f_{\lambda,n})=\frac7{60},~~\delta(f_{\lambda,n})=\frac1{12},$$
by Theorem \ref{thmlambdakappa}.

 For each integer $n>0$, the family $f_{\lambda,n}$ is the same as $f_\lambda$ except for a finite number of fibers. So the image of $f_{\lambda,n}$ in $\overline{\CM}_g$ induced by the moduli map $J_{f_{\lambda,n}}:\mP^1\to \overline{\CM}_g$ is the same as that of $f_\lambda$.
 Hence we will complete our proof after proving Claim A.\\

Proof of Claim A:  See Figure \ref{Branchlambda} for the branch locus $R_\lambda$ in $P$.

\begin{figure}[h]
\begin{center}
\setlength{\unitlength}{1mm}
\begin{picture}(50,53)(0,0)

\put(0,5){\line(1,0){55}}
\put(56,5){\makebox(0,0)[l]{$\scriptstyle{(\mathbb{P}^1,~t)}$}}
\put(5,2){\makebox(0,0)[l]{$\scriptstyle{-1}$}}
\put(25,2){\makebox(0,0)[l]{$\scriptstyle{0}$}}
\put(45,2){\makebox(0,0)[l]{$\scriptstyle{\infty}$}}

\multiput(5,10)(0,2){20}{\line(0,1){1.6}}
\put(0,10){\makebox(0,0)[l]{$\scriptscriptstyle{\Gamma_{-1}}$}}

\put(10,15){\oval(10,10)[l]} \qbezier(5,19)(5.8,19.5)(6.6,19)
\put(5.5,21){\makebox(0,0)[l]{$\scriptscriptstyle{5}$}}
\put(-4.5,15){\makebox(0,0)[l]{$\scriptscriptstyle{(-1,-1)}$}}

\put(3,29){\line(1,0){4}}

\multiput(25,10)(0,2){20}{\line(0,1){1.6}}
\put(20,10){\makebox(0,0)[l]{$\scriptscriptstyle{\Gamma_0}$}}

\qbezier(26.5,34)(23.5,30)(26.5,26) \qbezier(25,33)(25.5,34)(26,33)
\put(25.5,35){\makebox(0,0)[l]{$\scriptscriptstyle{3}$}}

\qbezier(27.3,31)(22.8,30)(27.3,29)
\qbezier(25.2,31)(25.7,32)(26.2,31)
\put(26,32){\makebox(0,0)[l]{$\scriptscriptstyle{2}$}}

\put(23.5,30){\line(1,0){4}}

\put(17.5,30){\makebox(0,0)[l]{$\scriptscriptstyle{(0,0)}$}}

\put(45,10){\line(0,1){40}}
\put(40,10){\makebox(0,0)[l]{$\scriptscriptstyle{\Gamma_{\infty}}$}}

\qbezier(46.5,49)(43.5,45)(46.5,41) \qbezier(45,48)(45.5,49)(46,48)
\put(45.5,50){\makebox(0,0)[l]{$\scriptscriptstyle{2}$}}

\qbezier(47.3,46)(42.8,45)(47.3,44)
\qbezier(45.2,46)(45.7,47)(46.2,46)
\put(46,47){\makebox(0,0)[l]{$\scriptscriptstyle{2}$}}

\qbezier(42.7,46)(47.2,45)(42.7,44)
\qbezier(44.8,46)(44.3,47)(43.8,46)
\put(43,47){\makebox(0,0)[l]{$\scriptscriptstyle{2}$}}

\put(35,42){\makebox(0,0)[l]{$\scriptscriptstyle{(\infty,\infty)}$}}
\end{picture}
\end{center}
\vspace{-0.2cm}
\caption{Branch locus of $f_\lambda$\label{Branchlambda}}
\end{figure}
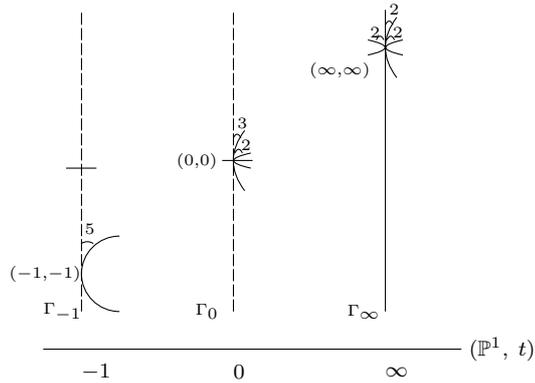

Denote by $r_i(R_h)$ the contribution of the point $(x,t)=(i,i)~(i=-1,0,\infty)$
to the relative ramification index $r(R_h)$.\\

\noindent \underline{$F_0$}: Let $p$ be the point $(x,t)=(0,0)$. The local equation of $R_h$
near $p$ is $h(x,t)$ in (\ref{hxt}), thus

(1) The root $x=0$ of $h(x,0)=0$ is with multiplicity 6.

(2)  The  point $p$ is a singularity of $R_h$ with multiplicity $m_1=3$,
and the vertical direction is a tangent line of $R_h$ with multiplicity 2.

(3) From the following figure of resolution of $p$, we have that
$r_0(R_h)=m_1(m_1-1)+m_2(m_2-1)+3=11$, where 3 comes from the contribution of smooth ramification points (see (\ref{ramindex})).
\vspace{-8mm}
\begin{center}
\setlength{\unitlength}{1.5mm}
\begin{picture}(50,20)(0,0)

\put(-3,5){\makebox(0,0)[l]{$\scriptstyle{p}$}}
\multiput(0,0)(0,1.5){8}{\line(0,1){1}}
\put(-1,5){\line(1,0){3}}
\qbezier(2,8)(-2,5)(2,2) \qbezier(2,6.5)(-2,5)(2,3.5)

\put(5,5){${\stackrel{\scriptstyle{\m_1=3}}{\longleftarrow}}$}

\multiput(14,0)(0,1.5){8}{\line(0,1){1}}
\put(12,3){\line(1,1){4}}
\multiput(12,5)(1.5,0){8}{\line(1,0){1}}
\qbezier(16,8)(12,5)(16,2)
\qbezier(14,7.5)(14.7,8)(15.4,7.5)
\put(14.5,9){\makebox(0,0)[l]{$\scriptscriptstyle{2}$}}
\qbezier(21,6.5)(22,5)(21,3.5)
\qbezier(43,9)(44,8)(45,9)

\put(27,5){${\stackrel{\scriptstyle{\m_2=2}}{\longleftarrow}}$}

\multiput(36,0)(0,1.5){8}{\line(0,1){1}}
\put(34,3){\line(1,1){4}}
\qbezier(40,6.5)(41,5)(40,3.5)
\multiput(34,5)(1.5,0){8}{\line(1,0){1}}
\multiput(44,0)(0,1.5){8}{\line(0,1){1}}
\end{picture}
\end{center}
\vspace{3mm}

Hence
the dual graph of $F_0$ is Figure (2-1a).\\

\noindent  \underline{$F_{-1}$}: The local equation of $R_h$ near $(x,t)=(-1,-1)$ is
\begin{align*}
&h_{-1}(u,s):=h(u-1,s-1)\\
=&u^6-6u^5+(15u^4-20u^3+60u^2-72u+32)s+(-45u^2+66u-36)s^2+5s^3.
\end{align*}
So $R_h$ is smooth near $(u,s)=(0,0)$, and $u=0$ is a root
of $f(u,0)=u^6-6u^5$ with multiplicity 5. Thus $r_{-1}(R_h)=4$.

The local equation of $R_h$ near $(x,t)=(-1,-1)$ is the same as $y^2=x^5+t$, and
 $F_{-1}$ is of type [\Rnm8-1] whose dual graph is Figure \ref{Dualgraphs}(a). (See Figure \ref{Dualgraphs}, where $\bullet$ denotes a smooth elliptic curve.)

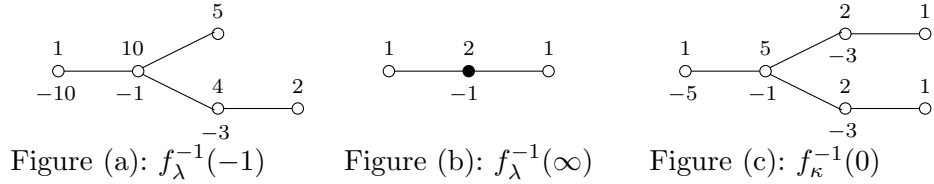
\begin{figure}[h]
 \setlength{\unitlength}{1.5mm}
\begin{center}
\begin{picture}(60,15)(10,-3)

\multiput(0,5)(7,0){2}{\circle{1}}
\put(0.5,5){\line(1,0){6}}
\put(7.4,4.7){\line(2,-1){6.3}}
\put(7.4,5.3){\line(2,1){6.3}}
\put(14,8.3){\circle{1}}
\put(14,1.7){\circle{1}}
\put(21,1.7){\circle{1}}
\put(14.5,1.7){\line(1,0){6}}
\put(5,3){\makebox(0,0)[l]{$\scriptstyle{-1}$}}
\put(-2,3){\makebox(0,0)[l]{$\scriptstyle{-10}$}}
\put(-0.5,7){\makebox(0,0)[l]{$\scriptstyle{1}$}}
\put(5.5,7){\makebox(0,0)[l]{$\scriptstyle{10}$}}
\put(13.5,10.3){\makebox(0,0)[l]{$\scriptstyle{5}$}}
\put(13.5,3.7){\makebox(0,0)[l]{$\scriptstyle{4}$}}
\put(12.5,-0.5){\makebox(0,0)[l]{$\scriptstyle{-3}$}}
\put(20.5,3.7){\makebox(0,0)[l]{$\scriptstyle{2}$}}
\put(7,-3){\makebox(0,0)[c]{\mbox{Figure~}(a):~$f_\lambda^{-1}(-1)$}}

\multiput(29,5)(7,0){3}{\circle{1}}
\multiput(29.5,5)(7,0){2}{\line(1,0){6}}
\put(36,5){\circle*{1}}
\put(28.5,7){\makebox(0,0)[l]{$\scriptstyle{1}$}}
\put(34.3,3){\makebox(0,0)[l]{$\scriptstyle{-1}$}}
\put(35.5,7){\makebox(0,0)[l]{$\scriptstyle{2}$}}
\put(42.5,7){\makebox(0,0)[l]{$\scriptstyle{1}$}}
\put(36,-3){\makebox(0,0)[c]{\mbox{Figure~}(b):~$f^{-1}_\lambda(\infty)$}}

\multiput(55,5)(7,0){2}{\circle{1}}
\put(55.5,5){\line(1,0){6}}
\put(62.4,4.7){\line(2,-1){6.3}}
\put(62.4,5.3){\line(2,1){6.3}}
\multiput(69,8.3)(7,0){2}{\circle{1}}
\put(69.5,8.3){\line(1,0){6}}
\multiput(69,1.7)(7,0){2}{\circle{1}}
\put(69.5,1.7){\line(1,0){6}}
\put(53.5,3){\makebox(0,0)[l]{$\scriptstyle{-5}$}}
\put(60.5,3){\makebox(0,0)[l]{$\scriptstyle{-1}$}}
\put(54.5,7){\makebox(0,0)[l]{$\scriptstyle{1}$}}
\put(61.5,7){\makebox(0,0)[l]{$\scriptstyle{5}$}}
\put(67.5,6.3){\makebox(0,0)[l]{$\scriptstyle{-3}$}}
\put(68.5,10.3){\makebox(0,0)[l]{$\scriptstyle{2}$}}
\put(75.5,10.3){\makebox(0,0)[l]{$\scriptstyle{1}$}}
\put(68.5,3.7){\makebox(0,0)[l]{$\scriptstyle{2}$}}
\put(75.5,3.7){\makebox(0,0)[l]{$\scriptstyle{1}$}}
\put(67.5,-0.3){\makebox(0,0)[l]{$\scriptstyle{-3}$}}
\put(62,-3){\makebox(0,0)[c]{\mbox{Figure~}(c):~$f_\kappa^{-1}(0)$}}

\end{picture}
\end{center}

\caption{Singular fibers in fibrations with minimal modular invariants\label{Dualgraphs}}
\end{figure}

\noindent  \underline{$F_\infty$}: Let $q$ be the point $(x,t)=(\infty,\infty)$.
The local equation of $R_h$ near $q$ is
\begin{equation}\label{lambdainftyeqn}
h_\infty(w,r):=w^6r^3h(\frac1w,\frac1r)
 =5w^6-(24w^5+45w^4)r+(40w^3+15w^2)r^2+r^3.
\end{equation}
Then we know that

(1) The root $w=0$ of $h_\infty(w,0)=0$ is of multiplicity 6.

(2) The point $q$ is a singularity of $R_h$ with multiplicity $m_1=3$,
and the vertical direction  is a tangent line of $R_h$ with multiplicity 3.

(3) From the following figure of resolution of $q$, we have that
$r_\infty(R_h)=m_1(m_1-1)+m_2(m_2-1)+3=15$.

\vspace{-8mm}
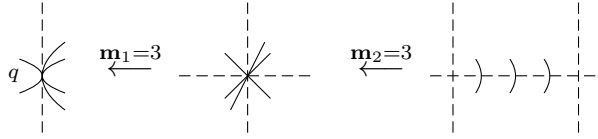
\begin{center}
\setlength{\unitlength}{1.5mm}
\begin{picture}(50,20)(0,0)

\put(-3,5){\makebox(0,0)[l]{$\scriptstyle{q}$}}
\multiput(0,0)(0,1.5){8}{\line(0,1){1}}
\qbezier(2,8)(-2,5)(2,2) \qbezier(2,6.5)(-2,5)(2,3.5)
\qbezier(-2,6.5)(2,5)(-2,3.5)

\put(5,5){${\stackrel{\scriptstyle{\m_1=3}}{\longleftarrow}}$}

\multiput(18,0)(0,1.5){8}{\line(0,1){1}}
\put(16,3){\line(1,1){4}}
\put(20,3){\line(-1,1){4}}
\put(16.5,2){\line(1,2){3}}
\multiput(12,5)(1.5,0){8}{\line(1,0){1}}

\put(27,5){${\stackrel{\scriptstyle{\m_2=3}}{\longleftarrow}}$}

\multiput(36,0)(0,1.5){8}{\line(0,1){1}}
\qbezier(38,6.5)(39,5)(38,3.5)
\qbezier(41,6.5)(42,5)(41,3.5)
\qbezier(44,6.5)(45,5)(44,3.5)
\multiput(34,5)(1.5,0){10}{\line(1,0){1}}
\multiput(47,0)(0,1.5){8}{\line(0,1){1}}

\end{picture}
\end{center}
\vspace{3mm}

Hence the local equation of $R_\lambda$ near $q$ is the same as $y^2=t\Pi_{i=1}^3(x^2+\alpha_it)$, $F_\infty$ is of type
$\mathrm{[\Rnm2]}$
and the dual graph of $F_\infty$ is Figure \ref{Dualgraphs}(b).\\

Now we know that the relative ramification of $R_h$ is
$$r(R_h)=K_{P/\mP^1}R_h+R_h^2=30\geq r_{-1}(R_h)+r_0(R_h)+r_\infty(R_h)=30.$$
So $f_\lambda$ has no other singular fibers.

\end{proof}

\begin{proposition}\label{propkappa}
There is a family of fibrations $(f_{\kappa,n}:X_n\to \mP^1)_{n\in\mN}$ of genus 2 with $\kappa(f_{\kappa,n})=\frac1{15},~\lambda(f_{\kappa,n})=\frac1{30},~\delta(f_{\kappa,n})=\frac1{3}$, satisfying that

1)  $f_{\kappa,n}$ has $2n+3$ singular fibers;

2) the  image of $f_{\kappa,n}$ in $\overline{\CM}_g$ by the moduli  map $J:\mP^1\to \overline{\CM}_g$ is the same as that of  $f_{\kappa,0}$, for each $n\in\mN$.

\end{proposition}

\begin{proof}
This proof is similar to that of Proposition \ref{proplambda}.

Let $\Gamma_t$ be the fiber over $t\in\mP^1$ of the second projection
$$p_2:P=\mathbb{P}^1\times\mathbb{P}^1\to \mP^1, ~~p_2((x,t))=t.$$
Let $R_g$ be the divisor on $P$,
whose affine equation is
\begin{equation}\label{gxt}
g(x,t)=5x^6-18x^5+(15x^4+20x^3)t+(-45x^2+30x-16)t^2+9t^3.
\end{equation}
Let  $R_{\kappa,n}=R_g+\Gamma_\infty+\sum_{i=1}^{2n}\Gamma_i$, where $n\geq0$ and $\Gamma_i$'s are generic fibers of $p_2$. Combining with the second projection $p_2$, let $f_{\kappa,n}:X_n\to \mP^1$ be the relative minimal fibration determined by the double cover over $P$ whose branch locus is $R_{\kappa,n}$.\\

Case 1: $n=0$.
Denote $f_{\kappa,0}$ (resp. $R_{\kappa,0}$) by $f_\kappa$ (resp. $R_\kappa$) for brief.\\

{\it Claim B: There are exactly three singular fibers $F_1=f^{-1}_\kappa(1)$, $F_{0}=f^{-1}_\kappa(0)$ and $F_{\infty}=f^{-1}_\kappa(\infty)$ in $f_\kappa$. Moreover, $F_1$ is the fiber in Theorem \ref{thmlambdakappa} (2), $F_{0}$ is of type [\Rnm9-1] and $F_\infty$ is of type [\Rnm2], see \cite{NU73} for the notations.}\\

Assume Claim B firstly, then we have that
$$\kappa(f_{\kappa})=\frac1{15},~~\delta(f_\kappa)=\frac1{3},$$
by Theorem \ref{thmlambdakappa}, for both $F_{0}$ and $F_\infty$ have smooth reduction.
Furthermore, $\lambda(f_\kappa)=\frac1{30}$ by Noether equality.
From Equation (\ref{modinv}) and the Chern numbers of fibers in Figure \ref{chenextremal} and Figure \ref{cheniso}, we know that
\begin{align*}
K^2_{f_{\kappa}}&=\kappa(f_\kappa)+c_1^2(F_1)+c_1^2(F_0)+c_1^2(F_\infty)=\frac1{15}+\frac{4}{3}+\frac85+2=5,\\
\chi_{f_{\kappa}}&=\lambda(f_\kappa)+\chi_{F_1}+\chi_{F_0}+\chi_{F_\infty}=\frac1{30}+\frac{2}{3}+\frac45+\frac12=2.
\end{align*}

Since $K^2_{f_\kappa}=5<6$, $X_0$ is a ruled surface by Theorem 0.2 in \cite{TTZ05}. Furthermore, $X_0$ is a rational surface for $q(X_0)=q_{f_\kappa}=0$ by Theorem \ref{thmlambdakappa}.\\

Case 2: $n\geq1$.

Comparing $R_{\kappa,n}$ with $R_\kappa$, it is easy to see that $f_{\kappa,n}$ has $2n+3$ singular fibers, three of them are the same as singular fibers of $f_\kappa$ and the others are all of type [$\mathrm{\Rnm1^*_{0-0-0}}$]. Hence
$$\lambda(f_{\kappa,n})=\frac1{30},~~\kappa(f_{\kappa,n})=\frac1{15},~~\delta(f_{\kappa,n})=\frac1{3},$$
by Theorem \ref{thmlambdakappa}.

 For each integer $n>0$, the family $f_{\kappa,n}$ is the same as $f_\kappa$ except for a finite number of fibers. So the image of $f_{\kappa,n}$ in $\overline{\CM}_g$ induced by the moduli map is the same as that of $f_\kappa$.
 Hence we will complete the proof after proving Claim B.\\

 Proof of Claim B:  See Figure \ref{Branchkappa} for the branch locus $R_\kappa$ in $P$.

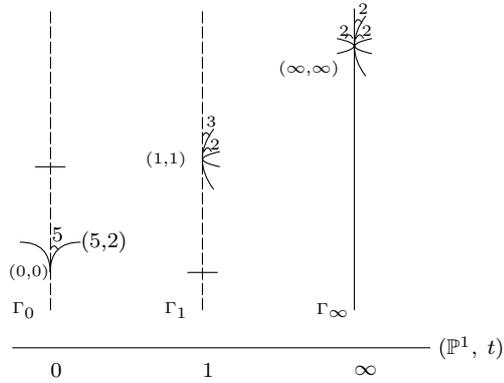
\begin{figure}[!h]
\begin{center}
\setlength{\unitlength}{1mm}
\begin{picture}(50,53)(0,0)

\put(0,5){\line(1,0){55}}
\put(56,5){\makebox(0,0)[l]{$\scriptstyle{(\mathbb{P}^1,~t)}$}}
\put(5,2){\makebox(0,0)[l]{$\scriptstyle{0}$}}
\put(25,2){\makebox(0,0)[l]{$\scriptstyle{1}$}}
\put(45,2){\makebox(0,0)[l]{$\scriptstyle{\infty}$}}

\multiput(5,10)(0,2){20}{\line(0,1){1.6}}
\put(0,10){\makebox(0,0)[l]{$\scriptscriptstyle{\Gamma_0}$}}

\qbezier(5,15)(5,19)(1,19) \qbezier(5,15)(5,19)(9,19)
\qbezier(5,18)(5.5,19)(6,18)
\put(5.2,20){\makebox(0,0)[l]{$\scriptstyle{5}$}}
\put(9,19){\makebox(0,0)[l]{$\scriptstyle{(5,2)}$}}

\put(-0.5,15){\makebox(0,0)[l]{$\scriptscriptstyle{(0,0)}$}}

\put(3,29){\line(1,0){4}}

\multiput(25,10)(0,2){20}{\line(0,1){1.6}}
\put(20,10){\makebox(0,0)[l]{$\scriptscriptstyle{\Gamma_1}$}}

\qbezier(26.5,34)(23.5,30)(26.5,26) \qbezier(25,33)(25.5,34)(26,33)
\put(25.5,35){\makebox(0,0)[l]{$\scriptscriptstyle{3}$}}

\qbezier(27.3,31)(22.8,30)(27.3,29)
\qbezier(25.2,31)(25.7,32)(26.2,31)
\put(26,32){\makebox(0,0)[l]{$\scriptscriptstyle{2}$}}

\put(23,15){\line(1,0){4}}

\put(17.5,30){\makebox(0,0)[l]{$\scriptscriptstyle{(1,1)}$}}

\put(45,10){\line(0,1){40}}
\put(40,10){\makebox(0,0)[l]{$\scriptscriptstyle{\Gamma_{\infty}}$}}

\qbezier(46.5,49)(43.5,45)(46.5,41) \qbezier(45,48)(45.5,49)(46,48)
\put(45.5,50){\makebox(0,0)[l]{$\scriptscriptstyle{2}$}}

\qbezier(47.3,46)(42.8,45)(47.3,44)
\qbezier(45.2,46)(45.7,47)(46.2,46)
\put(46,47){\makebox(0,0)[l]{$\scriptscriptstyle{2}$}}

\qbezier(42.7,46)(47.2,45)(42.7,44)
\qbezier(44.8,46)(44.3,47)(43.8,46)
\put(43,47){\makebox(0,0)[l]{$\scriptscriptstyle{2}$}}

\put(35,42){\makebox(0,0)[l]{$\scriptscriptstyle{(\infty,\infty)}$}}

\end{picture}
\end{center}
\caption{Branch locus of $f_\kappa$\label{Branchkappa}}
\end{figure}

Denote by $r_i(R_g)$ the contribution of the point $(x,t)=(i,i)~(i=0,1,\infty)$
to the relative ramification index $r(R_g)$.\\

\noindent \underline{$F_0$}: Let $p$ be the point $(x,t)=(0,0)$. The local equation of $R_g$
near $p$ is $g(x,t)$ in (\ref{gxt}), then we know that

(1) The root $x=0$ of $g(x,0)=0$ is with multiplicity 5.

(2)  The point $p$ is a singularity of $R_g$ with multiplicity $m_1=2$,
and the vertical direction is a tangent line of $R_g$ with multiplicity 2.

(3) From the following figure of resolution of $p$, we have that
$r_0(R_g)=m_1(m_1-1)+m_2(m_2-1)+4=8$, where 4 comes from the contribution of smooth ramification points.
\vspace{-8mm}
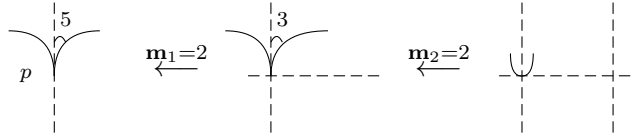
\begin{center}
\setlength{\unitlength}{1.5mm}
\begin{picture}(50,20)(0,0)

\put(-8,5){\makebox(0,0)[l]{$\scriptstyle{p}$}}
\multiput(-5,0)(0,1.5){8}{\line(0,1){1}}

\qbezier(-5,5)(-5,9)(-9,9) \qbezier(-5,5)(-5,9)(-1,9)
\qbezier(-5,8)(-4.5,9)(-4,8)
\put(-4.5,10){\makebox(0,0)[l]{$\scriptstyle{5}$}}
\put(3,5){${\stackrel{\scriptstyle{\m_1=2}}{\longleftarrow}}$}

\multiput(14,0)(0,1.5){8}{\line(0,1){1}}
\multiput(12,5)(1.5,0){8}{\line(1,0){1}}
\qbezier(14,5)(14,9)(10,9) \qbezier(14,5)(14,9)(19,9)
\qbezier(14,8)(14.5,9)(15,8)
\put(14.5,10){\makebox(0,0)[l]{$\scriptstyle{3}$}}

\put(26,5){${\stackrel{\scriptstyle{\m_2=2}}{\longleftarrow}}$}

\multiput(36,0)(0,1.5){8}{\line(0,1){1}}
\qbezier(36,5)(37,5)(37,7)\qbezier(36,5)(35,5)(35,7)

\multiput(34,5)(1.5,0){8}{\line(1,0){1}}
\multiput(44,0)(0,1.5){8}{\line(0,1){1}}
\end{picture}
\end{center}
\vspace{3mm}

Hence the local equation of $R_g$ near $p$ is the same as $y^2=x^5+t^2$, $F_0$ is of type $\mathrm{[\Rnm9-1]}$
and the dual graph of $F_0$ is Figure \ref{Dualgraphs}(c).\\

\noindent \underline{$F_1$}: Let $q$ be the point $(x,t)=(1,1)$.
The local equation of $R_g$ near $q$ is
\begin{equation*}
g_1(u,s):=g(u+1,s+1)
 =5u^6+12u^5+(15u^4+80u^3+60u^2)r-(45u^2+60u+4)r^2+9r^3.
\end{equation*}
Then we know that

(1) The root $u=0$ of $g_1(u,0)=0$ is of multiplicity 5.

(2) The point $q$ is a singularity of $R_g$ with multiplicity $m_1=2$,
and the vertical direction  is a tangent line of $R_g$ with multiplicity 2.

(3) From the following figure of resolution of $q$, we have that
$r_1(R_g)=m_1(m_1-1)+m_2(m_2-1)+3=7$,
where $3$ comes from the contribution of smooth ramification points.

\vspace{-8mm}
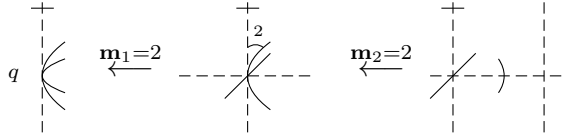
\begin{center}
\setlength{\unitlength}{1.5mm}
\begin{picture}(50,20)(0,0)

\put(-3,5){\makebox(0,0)[l]{$\scriptstyle{q}$}}
\multiput(0,0)(0,1.5){8}{\line(0,1){1}}
\put(-1,11){\line(1,0){2}}
\qbezier(2,8)(-2,5)(2,2) \qbezier(2,6.5)(-2,5)(2,3.5)

\put(5,5){${\stackrel{\scriptstyle{\m_1=2}}{\longleftarrow}}$}

\multiput(18,0)(0,1.5){8}{\line(0,1){1}}
\put(17,11){\line(1,0){2}}
\put(16,3){\line(1,1){4}}
\multiput(12,5)(1.5,0){8}{\line(1,0){1}}
\qbezier(20,8)(16,5)(20,2)
\qbezier(18,7.5)(18.7,8)(19.4,7.5)
\put(18.5,9){\makebox(0,0)[l]{$\scriptscriptstyle{2}$}}

\put(27,5){${\stackrel{\scriptstyle{\m_2=2}}{\longleftarrow}}$}

\put(35,11){\line(1,0){2}}
\multiput(36,0)(0,1.5){8}{\line(0,1){1}}
\put(34,3){\line(1,1){4}}
\qbezier(40,6.5)(41,5)(40,3.5)
\multiput(34,5)(1.5,0){8}{\line(1,0){1}}
\multiput(44,0)(0,1.5){8}{\line(0,1){1}}
\end{picture}
\end{center}
\vspace{3mm}

Hence the local equation of $R_g$ near $q$ is the same as $y^2=(x^2+t)(x^3+t)$,
and the dual graph of $F_1$ is Figure  (2-0a).\\

\noindent \underline{$F_\infty$}: The local equation of $R_g$ near $(x,t)=(\infty,\infty)$ is
\begin{align}\label{kappainftyeqn}
\begin{split}
g_\infty(w,r)&=w^6r^3g(1/w,1/r)\\
 &=9w^6+(-16w^6+30w^5-45w^4)r+(20w^3+15w^2)r^2+(-18w+5)r^3.
\end{split}
 \end{align}
It is easy to see that this is the same as $F_\infty$ in Proposition \ref{proplambda}.
In particular, $r_\infty(R_g)=15$.\\

Now we know that the relative ramification of $R_g$ is
$$r(R_g)=K_{P/\mP^1}R_g+R_g^2=30\geq r_0(R_g)+r_1(R_g)+r_\infty(R_g)=30.$$
So $f_\kappa$ has no other singular fibers.

\end{proof}

\section*{Acknowledgement}

 The authors would like to thank Prof. Shouwu Zhang for his helpful comments. They are very grateful to Prof. Jun Lu for his discussion for a long time. They also thank Prof. Tadashi Ashikaga, Prof. Kazuhiro Konno, and Prof. Yukio Matsumoto for useful comments on pseudo-periodic maps. This work was funded by the National Key Research and Development Program of China (grant 2018AAA0101001), the National Natural Science Foundation of China (grants 11601504, 11731004, and 11761141005), the Shanghai Science and Technology Commission Foundation (grants 18dz2271000 and 20511100200) and Fundamental Research  Funds of the Central Universities (No. DUT18RC(4)065). They would like to thank the referees sincerely for pointing out mistakes and useful detailed suggestions.


{ School of Mathematical Sciences, Dalian University of Technology, Dalian, Liaoning Province, P. R. of China.

 {\it E-mail address}: xlliu1124@dlut.edu.cn}

{School of Mathematical Sciences, Shanghai Key Laboratory of Pure Mathematics and Mathematical Practice, East China Normal University,
 Shanghai, China

 {\it E-mail address}:
 sltan@math.ecnu.edu.cn}

\clearpage

\end{document}